\documentclass[11pt, reqno, a4paper]{amsart}

\usepackage{amssymb}
\usepackage{amsthm}
\usepackage{indentfirst}
\usepackage{amsmath}
\usepackage{amscd} 
\usepackage{tikz-cd} 
\usepackage{mathabx} 
\usepackage{stmaryrd} 
\usepackage[backref=page]{hyperref}  
\usepackage{graphicx}       
\usepackage{psfrag}         
\usepackage[small]{caption} 

\makeatletter

\def\myMRbibitem{\@ifnextchar[\my@lbibitem\my@bibitem}

\def\mybiblabel#1#2{\@biblabel{{\hyperref{http://www.ams.org/mathscinet-getitem?mr=#1}{}{}{#2}}}}

\def\myhyperanchor#1{\Hy@raisedlink{\hyper@anchorstart{cite.#1}\hyper@anchorend}}

\def\my@lbibitem[#1]#2#3#4\par{%
    \item[\mybiblabel{#2}{#1}\myhyperanchor{#3}\hfill]#4%
    \@ifundefined{ifbackrefparscan}{}{\BR@backref{#3}}%
    \if@filesw{\let\protect\noexpand\immediate
       \write\@auxout{\string\bibcite{#3}{#1}}}\fi\ignorespaces%
}

\def\my@bibitem#1#2#3\par{%
    \refstepcounter\@listctr
    \item[\mybiblabel{#1}{\the\value\@listctr}\myhyperanchor{#2}\hfill]#3%
    \@ifundefined{ifbackrefparscan}{}{\BR@backref{#2}}%
    \if@filesw\immediate\write\@auxout
        {\string\bibcite{#2}{\the\value\@listctr}}\fi\ignorespaces%
}

\makeatother

\renewcommand*{\backref}[1]{}
\renewcommand*{\backrefalt}[4]{\quad \tiny 
    \ifcase #1 (Not cited.)%
    \or        (Cited on page~#2.)%
    \else      (Cited on pages~#2.)%
    \fi}

\theoremstyle{plain}
	\newtheorem{theo}{Theorem}
	\newtheorem{lemm}{Lemma}[section]
	\newtheorem{othe}[lemm]{Theorem}
	\newtheorem{coro}[lemm]{Corollary}
	\newtheorem{prop}[lemm]{Proposition}

\theoremstyle{definition}
	\newtheorem{clai}[lemm]{Claim}

	\newtheorem{ques}{Question}
	\newtheorem*{ack}{Acknowledgement}
\theoremstyle{remark}
	\newtheorem{exam}[lemm]{Example}

\newtheoremstyle{note}
  {5pt}
  {5pt}
  {\footnotesize}
  {}
  {\bfseries}
  {.}
  {.5em}
  {}

\theoremstyle{note}
	\newtheorem{rema}[lemm]{Remark}

  \def\CC{{\mathbb C}}

 \def\NN{{\mathbb N}}  
 \def\RR{{\mathbb R}}  
   
 \def\ZZ{{\mathbb Z}}

\def\la{\lambda}

\newcommand{\cC}{\mathcal{C}}
\newcommand{\cF}{\mathcal{F}}
\newcommand{\cG}{\mathcal{G}}

\newcommand{\cO}{\mathcal{O}}

\newcommand{\cU}{\mathcal{U}}
\newcommand{\cV}{\mathcal{V}}

\renewcommand{\epsilon}{\varepsilon}
\renewcommand{\rho}{\varrho}
\renewcommand{\phi}{\varphi}

\def\Diff{\operatorname{Diff}}
\def\Homeo{\operatorname{Homeo}}

\def\dim{\operatorname{dim}}

\def\supp{\operatorname{supp}}

\renewcommand{\setminus}{\smallsetminus}

\renewcommand{\angle}{\measuredangle}

\newcommand{\GL}{\mathrm{GL}}

\newcommand{\Ogroup}{\mathrm{O}}
\newcommand{\so}{\mathfrak{so}}
\newcommand{\Id}{\mathrm{Id}}

\DeclareMathOperator{\trace}{tr}

\newcommand{\flag}{\mathfrak{f}}    
\newcommand{\sflag}{\mathfrak{s}}   
\newcommand{\ordF}{\check{\mathcal{F}}} 

\newcommand{\sendtoouterspace}[1]{}

\hypersetup{bookmarksdepth = 3} 
\numberwithin{equation}{section}         

\begin{document}

\title[Robust Vanishing of Lyapunov Exponents for IFS's]{Robust Vanishing of All Lyapunov Exponents for Iterated Function Systems}

\author[Bochi]{Jairo Bochi}
\address{Departamento de Matem\'atica, Pontif\'{\i}cia Universidade Cat\'olica do Rio de Janeiro} 
\urladdr{\href{http://www.mat.puc-rio.br/~jairo}{www.mat.puc-rio.br/{\textasciitilde}jairo}}
\email{jairo@mat.puc-rio.br}

\author[Bonatti]{Christian Bonatti}
\address{Institut de Math\'ematiques de Bourgogne}
\email{bonatti@u-bourgogne.fr}

\author[D\'\i az]{Lorenzo J.~D\'\i az}
\address{Departamento de Matem\'atica, Pontif\'{\i}cia Universidade Cat\'olica do Rio de Janeiro} 
\email{lodiaz@mat.puc-rio.br}

\date{\today}

\thanks{The authors received support from CNPq, FAPERJ, PRONEX (Brazil), Balzan--Palis Project, Brazil--France Cooperation Program in Mathematics, and ANR (France).}

\begin{abstract}
Given any compact connected manifold $M$, we describe $C^2$-open sets of iterated functions systems (IFS's)
admitting fully-supported ergodic measures whose Lyapunov exponents along $M$ are all zero. 
Moreover, these measures are approximated by measures supported on periodic orbits.

We also describe $C^1$-open sets of IFS's admitting ergodic measures of positive entropy whose Lyapunov exponents along $M$ are all zero. 

The proofs involve the construction of non-hyperbolic measures 
for the induced IFS's on the flag manifold.
\end{abstract}

\maketitle

\section{Introduction}\label{s.intro}

\subsection{The hunt for (non-)hyperbolic measures}

Since the Multiplicative Ergodic Theorem of Oseledets \cite{Oseledets}, 
the Lyapunov exponents of invariant probability measures 
are central in differentiable dynamics.
As Oseledets reveals in the first paragraph of his celebrated paper, 
he was already interested in the dynamical implications of non-zero Lyapunov exponents.
Many of these implications, at least in the case of volume-preserving dynamics,
were discovered by Pesin during the mid-seventies (see e.g.~\cite{Pesin}).
Later, Katok~\cite{Katok} obtained strong consequences in the non-conservative case.
Roughly speaking, the absence of zero Lyapunov exponents permits to recover many
dynamical properties from uniformly hyperbolicity.
We refer the reader to the book \cite{BarreiraP} for much information
about the dynamics of systems without zero Lyapunov exponents, 
which are called \emph{nonuniformly hyperbolic}.

However, nonuniform hyperbolicity is not necessarily ubiquitous.
In the conservative situation, KAM theory gives rise 
to elliptic behavior which is robust in high regularity.
In lower regularity, zero Lyapunov exponents can also occur generically (see \cite{Bochi}).

\medskip

Outside the conservative setting,
we consider the general problem of determining which
are the possible Lyapunov spectra of the ergodic invariant probabilities of 
a given dynamical system.

Topological--geometrical properties of the dynamics impose restrictions on
the Lyapunov exponents; to give an obvious example, if the system is 
uniformly hyperbolic (say, Anosov) then no zero Lyapunov exponents can occur.\footnote{A more sophisticated relation of this kind was obtained by Johnson, Palmer and Sell \cite{JPS}.}
The converse of the implication above is false:
there exist smooth systems whose Lyapunov exponents (with respect to all
ergodic invariant measures) are uniformly bounded away from $0$,
but are not uniformly hyperbolic: see \cite[Remark~1.1]{BaBoSch}, \cite{CLR}. 
Anyway, these examples seem to be very special, 
and it is natural to ask in what generality the 
lack of uniform hyperbolicity of a system forces the appearance of non-hyperbolic measures, 
that is, measures admitting at least one vanishing Lyapunov exponent. 
We are especially interested in the case that non-hyperbolic measures occur in a robust way,
and we would like to understand further properties of those measures
(e.g.\  multiplicity of zero exponents, support, approximation by periodic orbits, entropy, etc.)

An important result in this direction was obtained by 
Kleptsyn and Nalksy \cite{KN}, who gave $C^1$-robust examples
of diffeomorphisms having ergodic non-hyperbolic measures. 
Their examples exist on any compact manifold of dimension at least $3$,
and are partially hyperbolic with integrable circle fibers.
The construction is based on their earlier paper joint with Gorodetski and Ilyashenko \cite{GIKN}, 
which obtains similar results for iterated function systems (IFS's) of the circle. 

These ideas have been used in \cite{DiGo} to
determine properties of homoclinic classes that 
imply the existence of non-trivial non-hyperbolic measures,
however under $C^1$-generic assumptions.
In \cite{BDG}, the construction was tuned to 
enlarge the supports of these measures:
they can be taken as the whole homoclinic class.

\medskip

The non-hyperbolic measures in all results above
are obtained as limits of sequences of measures supported on periodic orbits
whose central Lyapunov exponent converges to zero.
Therefore the non-hyperbolicity of the system is detected by its periodic orbits.
The general principle that periodic orbits carry a great amount of information about the dynamics
has been successful in many occasions; see e.g.\ \cite{Sigmund} in the uniformly hyperbolic context, 
\cite{Katok} for nonuniformly hyperbolic context, and \cite{Mane}, \cite{ABC} in the $C^1$-generic context.
It is thus natural to reformulate the previous problems 
focusing on the simplest class of invariant measures, namely those supported on periodic orbits.

Another common feature of the non-hyperbolic measures from the results above
is that they have only one vanishing Lyapunov exponent.
Since there are open sets of diffeomorphisms 
with nonhyperbolic subbundles of any given dimension, one wonders if these systems 
have ergodic measures with multiple zero exponents.
There is a clear difficulty in passing to higher dimensions:
we lose the commutativity of the products of central derivatives,
therefore also the losing the continuity of the exponents.
Those properties were crucial in the constructions above.

Finally, we observe that all these non-hyperbolic measures have zero entropy.

\medskip

In this paper, we commence the study of the robust existence of multiple zero exponents.
Following the strategy above, we first attack the simpler case of IFS's.
We extend the result of \cite{GIKN},
replacing the circle fiber by an arbitrary compact manifold $M$,
and finding ergodic measures whose Lyapunov exponents along $M$ 
all vanish. 
More precisely, we prove two parallel extensions of the \cite{GIKN} result:
\begin{enumerate}
	\item We construct $C^2$-open sets of IFS's having ergodic measures with only zero exponents along $M$,
	full support, and approximable in a strong sense by measures supported on periodic orbits.
	\item We construct $C^1$-open sets of IFS's having ergodic measures with only zero exponents along $M$, and positive entropy.
\end{enumerate}
However, these two extensions are non-intersecting: The measures from the first result have zero entropy, while the measures from the second one are not fully supported.
The first result provides a more explicit construction of the measures, while the second one is an indirect
existence theorem.

\subsection{Precise statements of the main results}

An \emph{iterated function system}, or \emph{IFS}, is simply a finite collection 
$G=(g_0, \dots, g_{\ell-1})$
of (usually continuous) self-maps of a (usually compact) space $M$.
Then we consider the semigroup generated by these transformations.
An IFS can be embedded in a single dynamical system,
the \emph{$1$-step skew-product} $\phi_G \colon\ell^\ZZ\times M\to \ell^\ZZ\times M$ over the full shift $\sigma$ on $\ell^\ZZ=\{0,\dots,\ell-1\}^\ZZ$,
which is defined by $\phi_G(\omega, p)= (\sigma(\omega), g_{\omega_0}(p))$.

From now on, the ambient $M$ will be a compact connected manifold without boundary of dimension $d$.
We will consider IFS's $G$ of diffeomorphisms of $M$.
Then, for any ergodic $\phi_G$-invariant 
measure $\mu$, Oseledets theorem associates its \emph{fibered Lyapunov exponents}, 
which are the values that can occur as limits
$$
\lim_{n \to +\infty} \frac{1}{n} \log \|D(g_{\omega_{n-1}} \circ \cdots \circ g_{\omega_0})(x) \cdot v \|,
\quad \text{(where $v \in T_x M \setminus \{0\}$)}
$$
for a positive measure subset of points $((\omega_n),x) \in \ell^\ZZ \times M$.

\medskip

Our first result is as follows:

\begin{theo}\label{t.main}
Let $M$ be a compact connected manifold without boundary.
Then there exist an integer $\ell \ge 2$ and 
an open set $\mathcal{U}$ in $(\Diff^2(M))^\ell$ 
such that for any $G = (g_0,\dots,g_{\ell-1})\in\cU$ 
the $1$-step skew-product $\phi_G$ has an ergodic invariant measure $\mu$
whose support is the whole $\ell^\ZZ \times M$
and whose fibered Lyapunov exponents all vanish.
Moreover, the measure $\mu$ is the weak-star limit of a sequence of 
$\phi_G$-invariant measures $\mu_n$, each of these supported on a periodic orbit. 
\end{theo}

As we will see, our strategy consists on proving a 
stronger version of Theorem~\ref{t.main},
concerning IFS's on flag manifolds -- see Theorem~\ref{t.mainflag}.

\medskip

Another main result is the following:

\begin{theo}\label{t.C1easy}
Let $M$ be a compact connected manifold without boundary.
Then there exist an integer $\ell \ge 2$ and 
an open set $\mathcal{V}$ in $(\Diff^1(M))^\ell$ 
such that for any $G \in\cV$ 
there exists a compact $\phi_G$-invariant set $\Lambda_G \subset \ell^\ZZ \times M$
with the following properties:
\begin{enumerate}
\item\label{i.easy1} All Lyapunov exponents (tangent to $M$)
of all invariant probabilities with support contained in $\Lambda$ are zero.
\item\label{i.easy2} The restriction of $\phi_G$ to $\Lambda_G$ has positive topological entropy.
\end{enumerate}
In particular, $\phi_G$ has 
an ergodic invariant measure with positive metric entropy and only zero fibered Lyapunov exponents.
\end{theo}

The last assertion follows immediately from the Variational Principle.

Compared to Theorem~\ref{t.main}, 
Theorem~\ref{t.C1easy} improves the robustness class from $C^2$ to $C^1$.
The non-hyperbolic measures produced by Theorem~\ref{t.C1easy} 
have the additional property of positive entropy,
but clearly do not have full support.
Moreover, we do not know if those measures can be approximated by measures supported on periodic orbits.

Theorem~\ref{t.C1easy} has a simpler proof than Theorem~\ref{t.main}.
The definition of the $C^2$-open set $\cU$ from Theorem~\ref{t.main} 
involves basically two conditions, ``maneuverability'' and ``minimality'',
while the $C^1$-open set $\cV$ from Theorem~\ref{t.C1easy} only requires maneuverability.
In particular, the sets $\cU$ and $\cV$ have nonempty intersection.

\subsection{Questions}\label{ss.questions}  

In view of our results extending \cite{GIKN}, 
it is natural to expect a corresponding generalization of \cite{KN},
that is, the existence of 
open examples of partially hyperbolic diffeomorphisms 
with multidimensional center so that there are measures 
all whose central exponents vanish.

\medskip

We list other questions, mainly about IFS's:

\begin{ques}
Are there $C^1$- or $C^2$-robust examples with non-hyperbolic measures of full support and positive entropy?
\end{ques}

\begin{ques}
It is possible to improve Theorem~\ref{t.main} so that the set of measures that satisfies the conclusions
is dense (or generic) in the weak-star topology? 
\end{ques}

Consider IFS's of volume-preserving or symplectic diffeomorphisms.
(See \cite{KoroN} for results and problems about such systems.)
The proof of Theorem~\ref{t.C1easy} can be easily adapted for the volume-preserving case.

\begin{ques}
Does the analogue of Theorem~\ref{t.main} hold true in conservative contexts?
A more interesting and difficult question is whether the measure $\mu$ in the theorem can be taken of the form
$\mu = \mu_0 \times m$, where $\mu_0$ is a shift-invariant measure and $m$ is the volume on the fibers $M$?
\end{ques}

\section{Outlines of the proofs}\label{s.ideas}

We first outline the proof of Theorem~\ref{t.main}, explaining the
main ingredients and difficulties of it, and discussing the novelties in comparison with \cite{GIKN}.
We also state a stronger result which implies Theorem~\ref{t.main}.

Later, we will explain how the tools developed to prove Theorem~\ref{t.main}
can be applied to yield the easier Theorem~\ref{t.C1easy}.

\subsection{Ergodic measures as limit of periodic measures} \label{ss.ergodic}

The starting point is Lemma~\ref{l.pretaporter} below,
which gives sufficient conditions for a sequence of invariant probability measures 
supported on periodic orbits to converge to an ergodic measure,
and also permits to determine the support of the limit measure.
Let us state this lemma precisely.

Let $h \colon N \to N$ be a homeomorphism of a compact metric space $N$,
and let $\cO'$ and $\cO$ be periodic orbits of $h$.
Let $\epsilon>0$ and $0<\kappa<1$.
We say that
\emph{$\cO'$ $\epsilon$-shadows $\cO$ during a proportion $1-\kappa$ of the time} if 
$$
\frac{1}{p'}
\# \left\{x' \in \cO' ; \; 
\text{there is } x \in \cO \text{ with } \max_{0\le i < p} d(h^i( x'), h^i( x)) < \epsilon \right\}
\ge 1-\kappa,
$$
where $p$ and $p'$ are the periods of $\cO$ and $\cO'$, respectively.
(Notice the asymmetry of the relation.)

\begin{lemm}[Limit of periodic measures]\label{l.pretaporter}
Fix a homeomorphism $h \colon N \to N$ of a compact metric space $N$.
Suppose $(\cO_n)$ is a sequence of periodic orbits of $h$ whose periods $p_n$ tend to infinity.
Suppose further that
the orbit $\cO_{n+1}$ $\epsilon_{n}$-shadows $\cO_{n}$ during a proportion $1-\kappa_n$ of the time,
where the sequences $\epsilon_n>0$ and $0<\kappa_n<1$ satisfy
$$
\sum_n \epsilon_n < \infty \quad \text{and} \quad \prod_n (1-\kappa_n) > 0.
$$
For each $n$, let $\nu_n$ be the invariant probability supported on $\cO_n$. 
Then the sequence $(\nu_n)$ converges in the weak-star topology to a measure $\nu$
that is ergodic for $h$ and 
whose support is given by
$$
\supp \nu = \bigcap_{n=1}^\infty \overline{\bigcup_{m=n}^{\infty} \cO_m} \,  .
$$ 
\end{lemm}

The lemma is just a rephrasing of Lemma~2.5 from \cite{BDG},
which in its turn is a refined version of Lemma~2 from \cite{GIKN}.

\subsection{The main difficulty with higher dimensions}\label{ss.difficulty}

We want to find a sequence $(\cO_n)$ of periodic orbits for the skew-product map $\phi_G$
that fits in the situation of Lemma~\ref{l.pretaporter}
and such that the resulting limit measure has the desired properties of zero (fibered) Lyapunov exponents 
and full support.

In the paper \cite{GIKN}, which deals with the one-dimensional case (i.e., $M$ is the circle),
the sequence of periodic orbits is constructed in such a way that the Lyapunov exponent
converges to zero. 
The construction is recursive: each new orbit $\cO_{n+1}$ 
is chosen in order to \emph{improve} the previous one
$\cO_{n}$, in the sense that the new Lyapunov exponent is closer to zero.
It is easy to modify their construction so to ensure that each new orbit is denser in
the ambient space, and thus, as we now know, obtain full support for the limit measure.
There are practically no requirements on starting orbit $\cO_1$: it needs only to be attracting.
We call this the \emph{bootstrapping procedure}, because it starts from nothing 
and by successive improvements eventually achieves its goal.
We will give more details about it later (\S~\ref{ss.boot}).

With Lemma~\ref{l.pretaporter} we can guarantee ergodicity and full support of the limit measure,
and so we are left to control its Lyapunov exponent.
In the one-dimensional situation of \cite{GIKN}, the Lyapunov exponent is given by an integral and so 
its dependence on the measure is continuous with respect to the weak-star topology.
Since the Lyapunov exponent along the sequence provided by the bootstrapping procedure 
converges to zero, we obtain a limit measure with zero Lyapunov exponent, as desired.

However, if $M$ has dimension $d > 1$ then
the Lyapunov exponents are no longer given by integrals.
Worse still, they can indeed be discontinuous as functions of the measure;
the best that can be said is that the top Lyapunov exponent is upper semicontinuous,
while the bottom Lyapunov exponent is lower semicontinuous.
So, even if all Lyapunov exponents along the orbit $\cO_n$ converge to zero as $n \to \infty$,
there is no guarantee that the limit measure will have zero Lyapunov exponents.

\begin{rema}\label{r.useless_semicont}
Usually, semicontinuity helps when we are trying to produce equal Lyapunov exponents (as e.g.\ in \cite{Bochi}). We could use semicontinuity here if we were able to apply Lemma~\ref{l.pretaporter} with 
$1-\kappa_n$ arbitrarily small, but this is not the case. Incidentally, we can apply the lemma with $\epsilon_n$ arbitrarily small (see the proof of Theorem~\ref{t.conditions} in \S~\ref{s.bootstrap}), but we make no use of this fact. 
\end{rema}

We overcome this difficulty by working with a skew-product on a larger space called the \emph{flag bundle}. This permits us to recover continuity of Lyapunov exponents and thus prove Theorem~\ref{t.main} by the same bootstrapping procedure. Passing to the flag bundle, however, has a price: we lose one order of differentiability, and this is basically why our results need $C^2$ regularity, as opposed to the $C^1$ regularity required by \cite{GIKN}.

\subsection{Flag dynamics}\label{ss.flag}

If $M$ is a compact manifold of dimension $d$,
we denote by $\cF M$ the \emph{flag bundle} of $M$,
that is, the set of $(x, F_1, \dots, F_d)$ where $x \in M$ and
$F_1 \subset \cdots \subset F_d$ are nested vector subspaces of the tangent space $T_x M$,
with $\dim F_i = i$.
Such a sequence of subspaces is called a \emph{flag} on $T_x M$.
Then $\cF M$ is a compact manifold, and the natural projection $\cF M \to M$ defines a fiber bundle.
Every $C^r$ diffeomorphism $g \colon M \to M$ can be lifted to a $C^{r-1}$ diffeomorphism
$\cF g \colon \cF M \to \cF M$ in the natural way, namely
$$
\cF g \colon (x, F_1, \dots, F_d) \mapsto \big( g(x) , Dg(x)(F_1), \dots, Dg(x)(F_d) \big) \, .
$$

Given an IFS on $M$
with set of generators $G = (g_0, \dots, g_{\ell-1}) \in \Diff^r(M)$, $r\ge 1$,
then we consider the IFS on the flag bundle $\cF M$ with set of generators 
$\cF G := (\cF g_0, \dots, \cF g_{\ell-1})$.
Corresponding to this new IFS, we have a $1$-step skew-product on $\ell^\ZZ \times \cF M$
which we will denote by $\cF \phi_G$.
Therefore we have the following commuting diagram, where then vertical arrows are the obvious projections:
\begin{equation}\label{e.2_skew_products}
\begin{tikzcd}
\ell^\ZZ \times \cF M \arrow{r}{\cF \phi_G} \arrow{d} &\ell^\ZZ \times \cF M \arrow{d}\\
\ell^\ZZ \times M     \arrow{r}{\phi_G} \arrow{d}     &\ell^\ZZ \times M     \arrow{d}\\
\ell^\ZZ \arrow{r}{\sigma}                            &\ell^\ZZ
\end{tikzcd}
\end{equation}

The main result we actually prove in this paper is the following:

\begin{theo}\label{t.mainflag}
Let $M$ be a compact connected manifold without boundary.
Then there exist an integer $\ell \ge 2$ and 
an open set $\mathcal{U}$ in $(\Diff^2(M))^\ell$ 
such that for any $G = (g_0,\dots,g_{\ell-1})\in\cU$ 
the $1$-step skew-product $\cF \phi_G$ has an ergodic invariant measure $\nu$
whose support is the whole $\ell^\ZZ \times \cF M$
and whose fibered Lyapunov exponents all vanish.
Moreover, the measure $\nu$ is the weak-star limit of a sequence of 
$\cF \phi_G$-invariant measures $\nu_n$, each of these supported on a periodic orbit.  
\end{theo}

Theorem~\ref{t.main} follows; let us see why.

If a probability measure $\nu$ on $\ell^\ZZ \times \cF M$ is $\cF\phi_G$-invariant and ergodic,
then its projection $\mu$ on $\ell^\ZZ \times M$ is $\phi_G$-invariant and ergodic.
As we will see later, the fibered Lyapunov exponents of (ergodic) $\nu$ can be expressed
as linear functions of the integrals 
\begin{equation}\label{e.I_j}
I_j (\nu) := 
\int_{\big( (\omega_n), x, (F_i) \big) \in \ell^\ZZ \times \cF M} 
\log \left| \det \left( Dg_{\omega_0}(x)|_{F_j}\right) \right|  \, d \nu.
\end{equation}

This has important consequences:
\begin{enumerate}
	\item\label{i.consequence_1} The fibered Lyapunov exponents of $\nu$ vary continuously with respect to $\nu$,
	among ergodic measures: if $\nu_n$ are ergodic measures converging to an ergodic measure $\nu$, then the fibered Lyapunov exponents of $\nu_n$ converge to those of $\nu$.
	\item\label{i.consequence_2} All the fibered Lyapunov exponents of (ergodic) $\nu$ vanish
if and only if all the fibered Lyapunov exponents of $\mu$ vanish.
\end{enumerate}

\begin{rema}
In fact, each fibered Lyapunov exponent of $\mu$ is also a fibered Lyapunov exponent of $\nu$
(see Example~\ref{ex.main}).
Notice that there is no contradiction with the aforementioned discontinuity of
the fibered Lyapunov exponents with respect to $\mu$.
Indeed,
a convergent sequence of $\phi_G$-invariant measures $\mu_n$  whose limit is ergodic
may fail (even after passing to a subsequence) 
to lift to a converging sequence of $\cF\phi_G$-invariant measures $\nu_n$ whose limit is ergodic.
\end{rema}

On one hand, property~(\ref{i.consequence_2})
makes Theorem~\ref{t.main} a corollary of Theorem~\ref{t.mainflag}.
On the other hand, property~(\ref{i.consequence_1}) extirpates the difficulty
explained before, and so allows us to prove Theorem~\ref{t.mainflag}
by the bootstrapping procedure,
as we explain next.

\subsection{The bootstrapping procedure on the flag bundle}\label{ss.boot}

An important feature of the bootstrapping procedure of \cite{GIKN}
is that each periodic orbit must be hyperbolic attracting (along the fiber);
this permits us to find each new orbit as a fixed point of a contraction.
So let us see how to detect contraction.

A linear isomorphism $L$ of $\RR^d$  
induces a diffeomorphism $\cF L$ of the corresponding manifold of flags.
If all eigenvalues of $L$ have different moduli, then the map $\cF L$ has a unique attracting fixed point.
This is instinctively clear; see Figure~\ref{f.flags}.
(Moreover, the map $\cF L$ is Morse--Smale, has $d!$ fixed points, and no periodic points other than these; see \cite{ShubVasquez}.)

\begin{figure}[ht]
\begin{center}
\includegraphics[width=.45\textwidth]{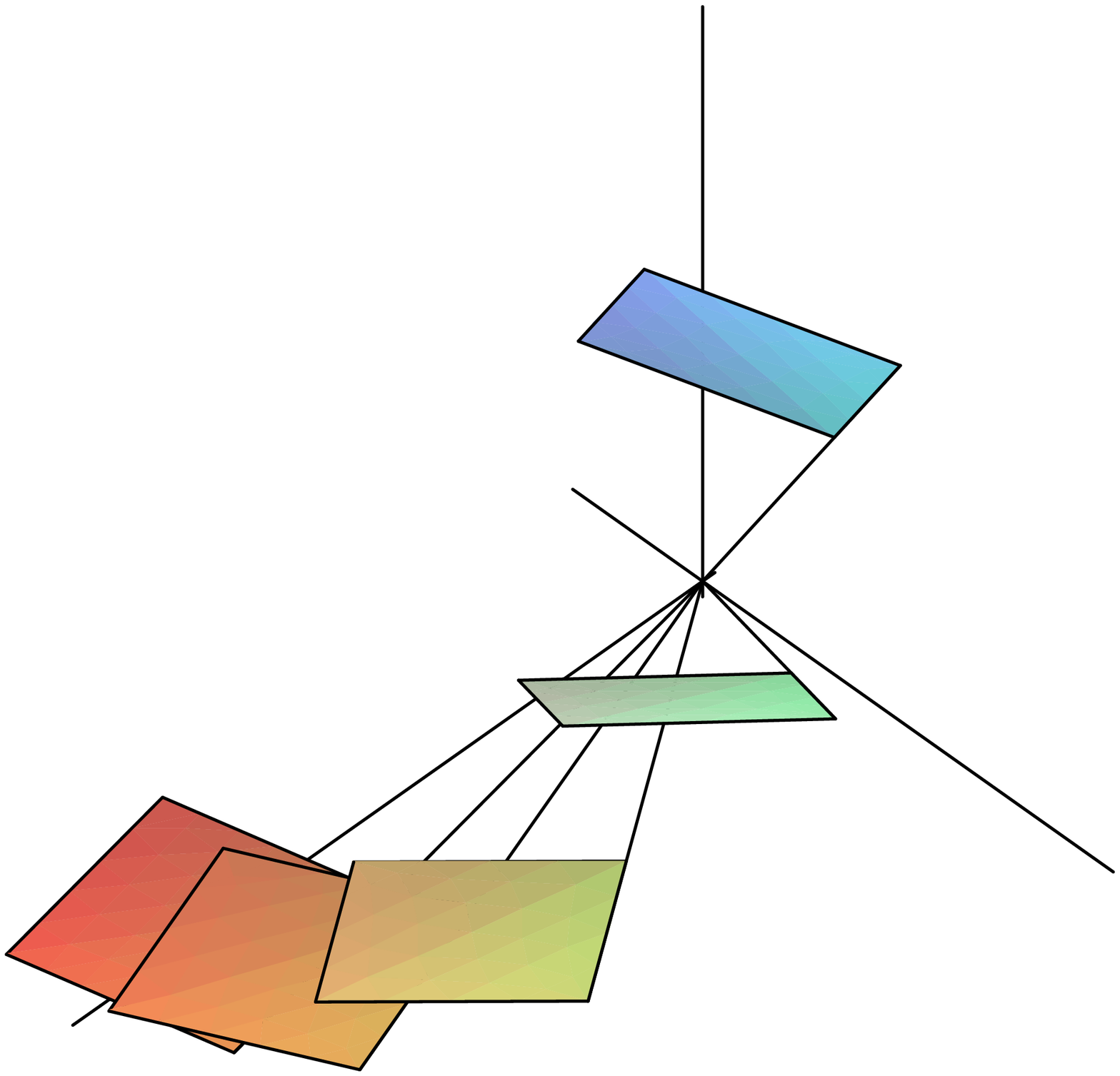}
\end{center}
\caption{A flag and its first four iterates under $\cF L$, where $L \colon \RR^3 \to \RR^3$ 
has eigenvalues $2$, $1$, $1/2$; the eigendirections are also pictured.}
\label{f.flags}
\end{figure}

Hence if $x$ is a fixed point of diffeomorphism $g \colon M \to M$
and the moduli of the eigenvalues of $Dg(x)$ are all different and less than $1$,
then there is a unique flag $\flag$ on $T_x M$ such that $(x,\flag)\in \cF M$ 
is an attracting fixed point for $\cF g$.
The converse is true: all hyperbolic attracting fixed points of $\cF g$ appear in this way.

Now let us sketch how to carry out a step in the bootstrapping procedure.
Of course, some conditions are needed for the IFS $G$;
we will see along the way how these conditions should look like.

Let us assume it is given a periodic point $(\omega, \xi) \in \ell^\ZZ \times \cF M$
for $\cF \phi_G$ whose fibered Lyapunov exponents are all negative, different, and close to zero.
Let $p$ be the period of the orbit.
Then $\omega$ consists of infinite repetitions of the word $w = \omega_0 \dots \omega_{p-1}$.
Then the ``point-flag'' $\xi$ is fixed and hyperbolic attracting for the diffeomorphism 
$h_1 := \cF g_{\omega_{p-1}} \circ \cdots \circ \cF g_{\omega_0}$.
Our aim is to find a new periodic point $(\tilde \omega, \tilde \xi)$
for $\cF \phi_G$ with the following properties:
\begin{itemize}
\item The orbit of $(\tilde \omega, \tilde \xi)$ closely shadows
the orbit of $(\omega, \xi)$ for most of the time.
\item The fibered Lyapunov exponents of the new orbit are still negative and different 
but closer to zero than those of the initial orbit.
\item The new orbit is ``denser'' in the ambient space than the previous orbit.
\end{itemize}
This is roughly done as follows (see also Figure~\ref{f.strategy}):
\begin{itemize}

\item Take a small ball $B_0$ around $\xi$ in $\cF M$.
Let $n$ be very large.
The ball $B_0$ is mapped by $h_1^n$ into a very small ball $B_1$ around $\xi$.

\item Then we select a long sequence of maps $g_{s_1}$, $g_{s_2}$, \dots, $g_{s_m}$ in  $G$
such that the derivative of 
$h_2 := \cF g_{s_m} \circ \cdots \circ \cF g_{s_1}$ at $\xi$ is strongly expanding.
This expansion, however, is not strong enough to compensate the previous contraction, so $h_2$ sends $B_1$ into a ball $B_2$ much bigger than $B_1$ but still much smaller than $B_0$.
(Actually the expansion factors must be chosen more carefully, but we will leave the details for later.)
All this is possible
if the set $G$ has a property that we call \emph{maneuverability}.
(See Section~\ref{s.conditions} for a precise definition.)

\item Next, we select maps $\cF g_{t_1}$, $\cF g_{t_2}$, \dots, $\cF g_{t_k}$,
such that: 
\begin{itemize}
	\item the union of successive images of the ball $B_2$ gets close to any point in $\cF M$
(we say that this orbit makes a ``tour'');
	\item the last image, which is $h_3(B_2)$ where $h_3 := \cF g_{t_k} \circ \cdots \circ \cF g_{t_1}$
is contained in $B_0$ (we say that the orbit ``goes home'').
\end{itemize}
The length $k$ of this part must be large, but it will be much smaller than either $n$ or $m$,
so there is plenty of space for $B_3$ to fit inside $B_0$.
(Actually the tour must be made on $\ell^\ZZ \times \cF M$, but this is not difficult to obtain.)
This ``\emph{tour and go home}'' phase is possible if the IFS $\cF G$ is \emph{positively minimal} on $\cF M$
(see \S~\ref{ss.prelim_IFS}).

\item Since the composed map $h_3 \circ h_2 \circ h_1^n$ sends the ball $B_0$ inside itself,
it has a fixed point $\tilde \xi$.
Using that the derivatives of the maps $\cF g_s$ are uniformly continuous,
we are able to show that $\tilde \xi$ is an attracting fixed point.
Moreover, we can show that the $h_2$ part
has the effect of making the Lyapunov exponents closer to zero.
The effect of $h_3$ in the Lyapunov exponents is negligible,
because the length $k$, despite big, is much smaller the length $pn+m$ of $h_2 \circ h_1^n$.

\item So we find the desired periodic point $(\tilde \omega, \tilde \xi)$,
where $\tilde \omega$ consists of infinite repetitions of the word 
$(\omega_0 \dots \omega_{p-1})^n s_1 \dots s_m t_1 \dots t_k$.

\end{itemize}

\psfrag{B}{\footnotesize $B_0$} 
\psfrag{C}{\footnotesize $B_1$} 
\psfrag{D}{\footnotesize $B_2$} 
\psfrag{E}{\footnotesize $B_3$} 
\psfrag{1}{\footnotesize $h_1^n$}
\psfrag{2}{\footnotesize $h_2$}
\psfrag{3}{\footnotesize $h_3$}
\begin{figure}[ht]
\begin{center}
\includegraphics[width=.6\textwidth]{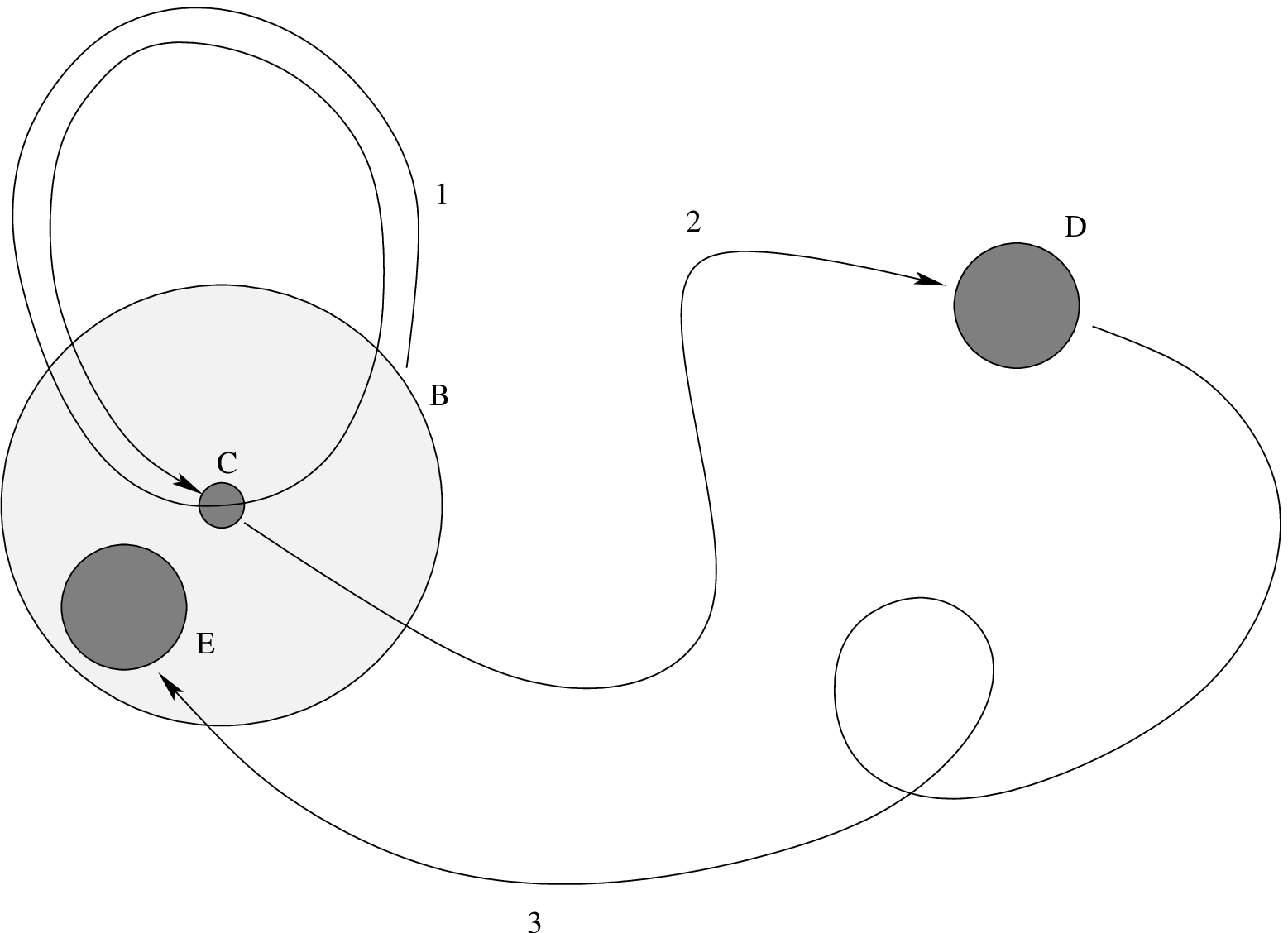}
\end{center}
\caption{A schematic picture of one bootstrapping step. 
The $B_i$'s are balls in the flag manifold $\cF M$. 
The map $h_2$ corresponds to the ``correcting'' phase; its effect is to approximate the exponents to zero.
The map $h_3$ corresponds to the ``tour and go home'' phase; its effect is to scatter the support and close the orbit.}
\label{f.strategy}
\end{figure}

We have sketched how the sequence of periodic orbits is produced.
Of course, the actual construction is more quantitative,
in order to guarantee that the Lyapunov exponents indeed converge to zero,
and that the requirements of Lemma~\ref{l.pretaporter} are indeed fulfilled.

The last and relatively easy step of the proof of Theorem~\ref{t.mainflag}
is to show that there is a nonempty open subset of $(\Diff^2(M))^\ell$
(for sufficiently large $\ell$) where the 
the prerequisites explained above 
(maneuverability and positive minimality on the flag manifold)
are satisfied.
This is done in Section~\ref{s.open}.

In conclusion, the proof of Theorem~\ref{t.mainflag} follows
an strategy very similar to that of \cite{GIKN}.
However, the control of the Lyapunov exponents is much more delicate
because the derivatives do not commute.
Here the flags come to our aid once again:
there is a distinctive feature of flag manifolds that 
permits us to put all the derivatives 
in a standard \emph{triangular} form, and 
therefore neutralize the non-commutativity effects.

\subsection{A $C^1$ construction with positive entropy but smaller support}

To prove Theorem~\ref{t.C1easy}, we use the maneuverability property
to construct an orbit in $\ell^\ZZ \times \cF M$ 
where the Birkhoff sums of the functions
$$
\big( (\omega_n), x, (F_i) \big) \in \ell^\ZZ \times \cF M 
\mapsto 
\log \left| \det \left( Dg_{\omega_0}(x)|_{F_j}\right) \right|
$$
(which are the integrands that appear in \eqref{e.I_j})
are uniformly bounded.
The compact set $\Lambda_G$ is taken as the projection on 
$\ell^\ZZ \times M$  of the closure of this orbit.
By what was seen above, this implies the zero exponents property.

Actually we impose some redundancy on the maneuverability property,
which easily implies positive topological entropy.

\subsection{Organization of the rest of paper}

Section~\ref{s.prelim} contains the preliminary definitions and properties.
In Section~\ref{s.conditions} we state explicit conditions 
(maneuverability and positive minimality on the flag manifold)
on the IFS $G$ 
that guarantee the existence of the measure $\nu$ satisfying the conclusions of Theorem~\ref{t.mainflag}.
In Section~\ref{s.bootstrap} we state Proposition~\ref{p.improve}, which makes precise the input and output of the recursive construction of periodic orbits; then, assuming this proposition, we prove that the bootstrapping procedure yields the desired results. In Sections~\ref{s.tour} and \ref{s.adjust} we prove technical consequences of minimality and maneuverability for later use. In Section~\ref{s.triang} we explain how maps on flag manifolds give rise naturally to triangular matrices, and why this is useful.
Section~\ref{s.proof} uses the material of all previous sections to prove the main Proposition~\ref{p.improve}. In Section~\ref{s.open} we prove that the existence of nonempty open sets of IFS's satisfying the prerequisites of maneuverability and positive minimality on the flag manifold. Finally, in Section~\ref{s.entropy} we prove Theorem~\ref{t.C1easy}.


\section{Preliminaries}\label{s.prelim}

In this section we collect definitions and basic properties 
about iterated function systems (IFS's), flag manifolds and bundles, and the related Lyapunov exponents.
Section~\ref{s.triang} gives deeper extra information
that is needed in the end of the proof.

\subsection{Iterated function systems}\label{ss.prelim_IFS}

If $N$ is a compact metric space and $h_0$, \dots, $h_{\ell-1}$ are homeomorphisms of $N$,
we denote by $\langle H \rangle$ the semigroup generated by $H$, i.e.,
the set of all maps $h_{s_m} \circ \dots \circ h_{s_1}$, where $s_1$, \dots, $s_m \in \{0,\dots,\ell-1\}$.
The concatenation $w = s_1 \dots s_m$ is called a \emph{word} of \emph{length} $m$ on the \emph{alphabet} $\{0,\dots,\ell-1\}$;
we then denote $h_{[w]} = h_{s_m} \circ \dots \circ h_{s_1}$.

An \emph{iterated function system} (or \emph{IFS}) is simply a semigroup $\langle H \rangle$
with a marked set $H$ of generators.

The \emph{$H$-orbit} of $x\in N$ 
is the set of the points $h(x)$ where $h$ runs on $\langle H \rangle$.
We say that $H$ (or $\langle H \rangle$) 
is \emph{positively minimal} 
if for every $x\in N$ the $H$-orbit of $x$ is dense in $M$.\footnote{Lemma~\ref{l.minimality} below
gives a practical criterion for positive minimality of an IFS.}

Let $\sigma$ be the shift transformation on 
the symbolic space $\ell^\ZZ = \{0,\dots,{\ell-1}\}^\ZZ$.
We define the \emph{$1$-step skew-product} 
$$
\phi_H \colon \ell^\ZZ \times N \to \ell^\ZZ \times N
$$ 
over $\sigma$
as $(\omega,y) \mapsto (\sigma(\omega), h_{\omega_0} (y))$, where $\omega_0$ is the zeroth symbol of the sequence $\omega$.

\begin{rema}
Let us mention some relations between positive minimality and the dynamics of the associated skew-product:
\begin{enumerate}
\item If the IFS $\langle H \rangle$ is positively minimal then 
$\phi_H$ is transitive on $\ell^\ZZ \times N$.
\item The IFS $\langle H \rangle$ is positively minimal if and only if 
for every point $z=(\omega,x) \in \ell^\ZZ \times N$, the union of the positive iterates of the 
\emph{local strong unstable manifold} 
$
W_{\mathrm{loc}}^{\mathrm{uu}}(z) := \big\{ (\tilde \omega, x) ; \; \tilde \omega_i = \omega_i \text{ for all } i<0 \big\}
$
under $\phi_H$
are dense in $\ell^\ZZ \times N$.
\end{enumerate}
We will not explicitly use these facts, so we omit the (easy) proofs.
\end{rema}

For further use, we endow the symbolic space $\ell^\ZZ$ with the distance
$$
d \big( (s_n), (t_n) \big) = 2^{-n_0}, \quad \text{where } n_0 = \min \{|n| ; \; s_n \neq t_n \}.
$$
On the product $\ell^\ZZ \times N$ we take the maximum of the distances of the two projections.

If $r$, $k$ are positive integers and 
$s_{-r}, s_{-r+1}, \dots, s_k$ are symbols in $\{0,\dots,\ell-1\}$,
then the \emph{cylinder}
$\llbracket s_{-r} \ldots ; s_0 \ldots s_k \rrbracket$
is the set of $(t_n)_n \in \ell^\ZZ$ such that $t_n = s_n$ for all $n$ with 
$-r \le n \le k$.
Cylinders $\llbracket s_{-r}\ldots ; \rrbracket$ and
$\llbracket ;s_0\ldots s_k \rrbracket$ are defined analogously.

\subsection{Lyapunov exponents}\label{ss.prelim_Lyap}

It is convenient to consider Lyapunov exponents in the general setting of bundle automorphisms.
Details can be found in the book \cite{LArnold}. 

\medskip

Let $X$ be a compact metric space. 
Let $V$ be a vector bundle of rank $d$ over $X$,
and let $\pi \colon V \to X$ be the bundle projection.
We endow $V$ with a Riemannian norm, that is, a continuous choice of an inner product 
on each fiber.

Let $T\colon X\to X$ be a continuous map.
Let $S \colon V \to V$ be a vector bundle morphism over $T$, i.e., 
a continuous map such that $\pi\circ S = T \circ \pi$ which is a linear map on each fiber. 
For $x \in X$ and $n \in \NN$, the restriction of $S^n$ to the fiber $V_x := \pi^{-1}(x)$
gives a linear map $A^{(n)}_S(x) \colon V_x \to V_{T^n x}$.
We write $A_S(x) = A^{(1)}_S(x)$ and so $A^{(n)}_S(x) = A_S(T^{n-1} x) \circ \cdots \circ A_S(x)$.

From now on, let us assume that each linear map $A(x)$ is an isomorphism.

In the case that the vector bundle is trivial
(i.e., $V = X \times E$ where $E$ is a vector space and $\pi$ is the projection on the second factor)
then the morphism $S$ is also called a \emph{linear cocycle},
and the map $A = A_S \colon X \to \GL(E)$ is called the \emph{generator} of the cocycle.

Sometimes, with some abuse of terminology, we also call a vector bundle morphism a \emph{cocycle}.

Let $\mu$ be a $T$-invariant ergodic probability measure on $X$.\footnote{All the measures we consider will be defined over the corresponding Borel $\sigma$-algebra.} 
By Oseledets Theorem, for $\mu$-almost every point $x\in X$ and every vector $v \in V_x \setminus \{0\}$,
the \emph{Lyapunov exponent}
$$
\lim_{n \to +\infty} \frac{1}{n} \log \|A^{(n)}(x) \cdot v \|,
$$
exists; moreover the ordered list of Lyapunov exponents $\lambda_1 \ge \dots \ge \lambda_d$
with repetitions according to multiplicity (i.e., the number of linearly independent vectors with the same exponent)
is almost everywhere independent of $x$.

\begin{exam}\label{ex.derivative_cocycle}
Suppose that $X = M$ is a smooth manifold of dimension $d$,
and $V = TM$ is the tangent bundle of $M$.
Let $T = g$ be a diffeomorphism of $M$, and let $S = Dg$ be the derivative of $g$.
This is sometimes called the ``\emph{derivative cocycle}''.
\end{exam}

Now let $M$ be a compact smooth manifold.
Given a continuous map $x \in X \mapsto g_x \in \Diff^1(M)$,
we consider the skew-product map $\phi$ on $X \times M$ defined by $\phi (x,y) = (T(x),g_x(y))$;
this is called a \emph{nonlinear cocycle}.
If $\nu$ is a probability on $X \times M$ that is $\phi$-invariant and ergodic,
then the \emph{fibered Lyapunov exponents} of the nonlinear cocycle $\phi$ with respect to $\nu$
are the values that can occur as limits
$$
\lim_{n \to +\infty} \frac{1}{n} \log \|D(g_{T^{n-1} x} \circ \cdots \circ g_x)(y) \cdot v \|,
\quad \text{(where $v \in T_y M \setminus \{0\}$)}
$$
for a positive measure subset of points $(x,y)\in X \times M$.
Of course, these are the previously defined Lyapunov exponents
relative to the action of the derivatives on the vector bundle $X \times TM$.

\begin{exam}\label{ex.step_skew_pr} 
Consider the $1$-step skew-product $\phi_H$ on $\ell^\ZZ \times N$ defined in \S~\ref{ss.prelim_IFS}.
If $N$ is a smooth manifold and each generator $h_s$ is a diffeomorphism
then $\phi_H$ can be viewed as a nonlinear cocycle,
and each ergodic invariant measure on $\ell^\ZZ \times N$ gives rise to 
fibered Lyapunov exponents.
\end{exam}

Analogously, we can also consider ``\emph{nonlinear cocycles}'' where the product $X \times M$ is replaced
by a fiber bundle over $X$ with typical fiber $M$ and structure group $\Diff^1(M)$.
For these nonlinear cocycles, we also consider the fibered Lyapunov exponents.

\subsection{Flag manifolds and linearly induced maps}\label{ss.prelim_flag}

Let $E$ be a real vector space of dimension $d$.
A \emph{flag} on $E$ is \label{def.flag}
a sequence $\flag = (F_i)_{i=1,\dots,d}$ 
of subspaces $F_1 \subset F_2 \subset \dots \subset F_d = E$
such that $\dim F_i = i$ for each $i$.
If each subspace $F_i$ is endowed with an orientation then 
we say that $\flag$ is an \emph{oriented flag}.
The set of flags (resp.\ oriented flags) on $E$ will be denoted by $\cF E$ (resp.\ $\ordF E$),
and called the \emph{flag manifold} (resp.\ \emph{oriented flag manifold}) of $E$; \label{def.flag_manifold}
indeed a differentiable structure is defined below.

Any flag (resp.\ oriented flag) $\flag = (F_i)$ on $E$ can be represented  
by a basis $(e_1, \dots, e_d)$ of $E$
such that for each $i$, $(e_1, \dots, e_i)$ is a basis (resp.\ positive basis) for $F_i$.
This representation is not unique.
However, if one fixes an inner product on $E$,
then each \emph{oriented} flag $\flag$ on $E$ has an unique orthonormal base that represents it;
this basis will be denoted by $\cO(\flag)$.\label{def.basis}

Thus one can endow the set $\ordF E$ with a structure of smooth manifold
diffeomorphic to $\Ogroup(d)$, the Lie group of $d \times d$ orthogonal matrices.
(More details are given in \S~\ref{ss.flag_derivative}.)
The disorientating mapping $\ordF E \to \cF E$ is $2^d$-to-$1$ covering map;
its deck transformations are smooth, and
therefore we can also endow $\cF E$ with a differentiable structure.
The manifolds $\cF E$ and $\ordF E$ are compact and have dimension $d(d-1)/2$;
the former is connected, while the latter has $2$ connected components.

If $E$ and $E'$ are vector spaces of the same dimension $d$,
then each linear isomorphism $L \colon E \to E'$ induces a map $\cF L \colon \cF E \to \cF E'$
in the obvious way:
$$
(F_1 \subset \cdots \subset F_d) \in \cF E \ \mapsto \ 
(L(F_1) \subset \cdots \subset L(F_d)) \in \cF E',
$$
By pushing-forward orientations, we define an analogous map $\ordF L \colon \ordF E \to \ordF E'$.
These two maps are actually diffeomorphisms;
more information about them will be given in \S~\ref{ss.flag_derivative}.

Let us fix some additional notation.
Suppose that $E$, $E'$ are vector spaces of the same dimension $d$, endowed with inner products.
If $L \colon E \to E'$ is a linear isomorphism and $\flag \in \ordF E$,
we let\label{def.M}
\begin{equation}\label{e.M}
M(L, \flag) := \text{the matrix of $L$ w.r.t.~the bases $\cO(\flag)$ and $\cO((\cF L)(\flag))$.}
\end{equation}
Notice that this is an upper triangular matrix,
whose diagonal entries $M_{11}$, \dots, $M_{dd}$ are positive and
satisfy the identity
\begin{equation}\label{e.subdeterminant}
M_{11} M_{22} \cdots M_{ii} = \det (L | F_i) \, .
\end{equation}

\begin{exam}\label{ex.QR_decomposition}
Suppose that $E=E'=\RR^d$ is endowed with the Euclidian inner product,
and identify the isomorphism $L$ with a $d \times d$ invertible matrix.
Suppose $\flag_0$ 
is the canonical flag in $\RR^d$ (i.e., that $\cO(\flag)$ is the canonical basis in $\RR^d$).
Consider the \emph{QR decomposition} of $L$, i.e., the unique factorization $L=QR$
where $Q$ is an orthogonal matrix and $R$ is an upper triangular matrix with positive diagonal entries.
(Those matrices are computed using the Gram--Schmidt process.)
Then $\cO((\cF L)(\flag_0))$ is the ordered basis formed by the columns of $Q$, 
and $M(L,\flag_0) = R$.
\end{exam}

If $\flag$ in a non-oriented flag 
then the entries of $M(L, \flag)$ are well-defined up to sign,
and the diagonal entries are well-defined and positive.

\subsection{Flag bundle dynamics}\label{ss.prelim_flag_dyn}

As is \S~\ref{ss.prelim_Lyap},
let $\pi \colon V \to X$ be a vector bundle of rank $d$ over a compact metric space $X$,
endowed with a Riemannian metric,
and let $S \colon V \to V$ be a vector bundle morphism over a continuous map $T \colon X \to X$
that is invertible in each fiber.

Let $\cF V$ be the \emph{flag bundle} associated to $V$,\label{def.flag_bundle}
that is, the fiber bundle over $X$ whose fiber $(\cF V)_x$ 
over $x\in X$ is the flag manifold of $V_x$.
The vector bundle morphism $S$
induces a fiber bundle morphism $\cF S$ of $\cF V$ also over~$T$.
This is summarized by the following diagrams:
$$
\begin{tikzcd}
V \arrow{r}{S} \arrow{d} &V \arrow{d}\\
X \arrow{r}{T}           &X
\end{tikzcd}
\qquad \leadsto \qquad
\begin{tikzcd}
\cF V \arrow{r}{\cF S} \arrow{d} &\cF V \arrow{d}\\
X     \arrow{r}{T}               &X
\end{tikzcd}
$$
Analogously we define the oriented versions $\ordF V$ and $\ordF S$.

\begin{rema}\label{r.oseledets_trick}
The original proof of 
Oseledets Theorem relies on this construction
to reduce the general case to the case of triangular cocycles:
see \cite[p.~228--229]{Oseledets}, also \cite[\S~4]{JPS}, \cite[\S~3.4.2]{BarreiraP}.
\end{rema}

\begin{exam}\label{ex.flag_derivative_cocycle}
Let us come back to the situation of Example~\ref{ex.derivative_cocycle}.
Consider the flag bundle $\cF(TM)$ associated to $TM$;
by simplicity we denote it by $\cF M$ and call it 
the \emph{flag bundle} of the manifold $M$. \label{def.flag_bundle_manifold}
It is a compact manifold of dimension $d(d+1)/2$,
and it is connected if $M$ is.
If $T = g$ is $C^r$ diffeomorphism of $M$, let $S = Dg$ be the derivative of $g$.
We obtain an induced morphism $\cF(Dg)$ of the flag bundle $\cF(TM) = \cF M$,
which by simplicity we denote by $\cF g$. \label{def.flag_derivative}
This morphism is a $C^{r-1}$ diffeomorphism of $\cF M$ 
(and a homeomorphism if $r=1$).
Analogously we define $\ordF M$ and $\ordF g$.
If $M$ is endowed with a Riemannian metric then $\ordF M$
can be naturally identified with the \emph{orthonormal frame bundle}.
The Riemannian metric on $M$ induces Riemannian metrics on $\cF M$ and $\ordF M$,
as explained in \S~\ref{ss.bundle_geometry}.
\end{exam}

Consider an ergodic probability measure $\nu$ for $\cF S$.
(Evidently, such measures always exist.)
We define the \emph{Furstenberg vector}\label{def.furstenberg} 
of $S$ with respect to $\nu$ 
as $\vec{\Lambda}(S,\nu) = (\Lambda_1,\dots,\Lambda_d)$, where
\begin{equation}\label{e.furstenberg}
\Lambda_j := \int_{\cF V} \log M_{i,i} (A_S(x), \flag) \, d\nu (x, \flag) \, .
\end{equation}
(Recall that $A_S(x)$ denotes the restriction of $S$ to the fiber $\pi^{-1}(x)$.)

In view of \eqref{e.subdeterminant}, we have:
\begin{equation}\label{e.sum}
\Lambda_1 + \dots + \Lambda_j = \int_{\cF V} \log \big| \det A_S(x)|_{F_j} \big|\, d\nu (x, \flag), 
\text{ where } \flag = (F_i). 
\end{equation}
Expressions like \eqref{e.furstenberg} and \eqref{e.sum} are called 
\emph{Furstenberg--Khasminskii formulas}; see \cite{LArnold}.

An obvious but important feature of the map $\nu\mapsto \vec{\Lambda}(\nu)$ is that it is continuous
with respect to the weak-star topology.

The next result relates the Furstenberg vector with 
the previously defined Lyapunov exponents:\footnote{The continuous-time version of Proposition~\ref{p.Lyap_permutation}
is sometimes called the \emph{Liao spectrum theorem}; see e.g.~\cite{Dai}.}

\begin{prop}\label{p.Lyap_permutation}  
Suppose $\nu$ is an ergodic probability measure for $\cF S$.
Let $\mu$ be the projection of $\nu$ on $X$ (thus an ergodic probability measure for $T$),
and let $\lambda_1$, \dots, $\lambda_d$ 
be the Lyapunov exponents of $S$ with respect to $\mu$.
Then there is a permutation $(k_1,k_2,\dots,k_d)$ of $(1, \dots, d)$ such that
the Furstenberg vector $(\Lambda_1,\dots,\Lambda_d)$ 
is given by $\Lambda_i = \lambda_{k_i}$.
\end{prop}

\begin{proof}
Using \eqref{e.sum}, the proposition follows from corresponding results for grassmannians;
see \cite[p.~265, 211]{LArnold}.
\end{proof}

It follows from Proposition~\ref{p.Lyap_permutation} that
the Furstenberg vector is independent of the choice of the Riemannian metric on the vector bundle $V$.\footnote{Of course, we can also prove this fact directly by showing that when the metric is changed the integrand in \eqref{e.furstenberg} (or \eqref{e.sum}) is replaced by a cohomologous one.}

The next result, which will be proved in \S~\ref{ss.lyapunovflag},
relates the fibered Lyapunov exponents of the nonlinear cocycle $\cF S$
with the Furstenberg vector:

\begin{prop}\label{p.Lyap_in_flag}
Suppose $\nu$ is an ergodic probability measure for $\cF S$,
and let $\vec\Lambda(\nu) = (\Lambda_1,\dots,\Lambda_d)$ 
be its Furstenberg vector.
Then the fibered Lyapunov exponents of $\cF S$ with respect to $\nu$
form the list of numbers 
$$
\Lambda_i - \Lambda_j \, , \quad \text{where $i<j$,}
$$
with repetitions according to multiplicity.
\end{prop}


\begin{exam}\label{ex.main}
Consider an IFS with set of generators $G = (g_0,\dots,g_{\ell-1}) \in (\Diff^2(M))^\ell$.
We can consider the iterated function system $\langle \cF G \rangle$
generated by $\cF G = (\cF g_0,\dots,\cF g_{\ell-1})$.  
The associated $1$-step skew-product $\cF \phi_{G} := \phi_{\cF G}$
fibers over $\phi_G$, i.e., the diagram \eqref{e.2_skew_products} commutes.
Let $\nu$ be an ergodic invariant probability for $\cF \phi_G$,
and let $\mu$ be its projection on $\ell^\ZZ \times M$.
Let us consider the fibered Lyapunov exponents of these nonlinear cocycles.
Let $\lambda_1 \ge \dots \ge \lambda_d$ be the fibered Lyapunov exponents of $\mu$.
Then there is a permutation $(k_1,k_2,\dots,k_d)$ of $(1, \dots, d)$ such that:
\begin{enumerate}
\item \label{i.thin} 
regarding $\cF \phi_G$ as a nonlinear cocycle over $\phi_G$, the fibered Lyapunov spectrum of $\nu$ is $\{\lambda_{k_i}-\lambda_{k_j} ; \; i< j\}$;
\item \label{i.fat}
regarding $\cF \phi_G$ as a nonlinear cocycle over $\sigma$, the fibered Lyapunov spectrum of $\nu$ is $\{\lambda_k\} \cup \{\lambda_{k_i}-\lambda_{k_j} ; \; i< j\}$.
\end{enumerate}
In particular, we see that Theorem~\ref{t.mainflag} implies Theorem~\ref{t.main}.
\end{exam}

\section{Sufficient conditions for zero exponents}\label{s.conditions}

In this section we state explicit conditions on an IFS $G = (g_0,\dots,g_{\ell-1}) \in (\Diff^1(M))^\ell$
that are sufficient for the existence of fully supported ergodic measures with zero exponents
as those in Theorem~\ref{t.mainflag}.

The first condition is that the IFS $\cF G$ of induced homeomorphisms of the flag manifold $\cF M$
is positively minimal on $\cF M$. 
For conciseness, we say that the IFS $G$ is \emph{positively minimal on the flag manifold}.

The next condition is this.
We say that a finite set $G \subset \Diff^1(M)$ has the {\emph{maneuverability} property if 
for every $(x,\flag)\in\cF M$ 
and for every sequence of signs $t = (t_1,\dots,t_d)\in\{-1,+1\}^d$, $d=\dim M$,
there is $g\in G$ such that 
$$
t_i \, \log M_{i,i}(Dg(x),\flag) > 0 \quad \text{for each $i$,}
$$
where $M_{i,i}(Dg(x),\flag)$ is the $i^{\text{th}}$ entry on the diagonal of the matrix $M(Dg(x),\flag)$
(defined in \S~\ref{ss.prelim_flag}).

Now we can state the following result:

\begin{othe}\label{t.conditions}
Consider a finite subset $G=\{g_0,\dots,g_{\ell-1}\}$ of $\Diff^2(M)$ 
with the following properties:
\begin{enumerate}
\item\label{i.cond_1} Positive minimality on the flag manifold.
\item\label{i.cond_2} Maneuverability.
\item\label{i.cond_3}
There is a map $g\in \langle G \rangle$ with a fixed point $x_0 \in M$
such that the eigenvalues of $Dg(x_0)$ are all negative, simple, and of different moduli.
\end{enumerate}
Then the skew-product map $\cF \phi_G$ 
possesses an ergodic invariant measure $\nu$ with support $\ell^\ZZ \times \cF M$
whose fibered Lyapunov exponents are all zero.
Moreover, the measure $\nu$ is the weak-star limit of a sequence of 
$\cF \phi_G$-invariant measures $\nu_n$, each of these supported on a periodic orbit.  
\end{othe}

\begin{rema}\label{r.C2}
Notice that the hypotheses of the theorem are meaningful for $C^1$ IFS's.
However, our proof requires $C^2$ regularity.
We do not know if the $C^1$ result is true.
\end{rema}

\begin{rema}
It is possible (by adapting arguments from Section~\ref{s.proof}) 
to show that Conditions~(\ref{i.cond_1}) and (\ref{i.cond_2}) actually imply Condition~(\ref{i.cond_3}),
and therefore the latter could be removed from the statement of Theorem~\ref{t.conditions}.
As our ultimate goal is to show the existence of the robust examples from Theorems~\ref{t.main} 
and \ref{t.mainflag}, we chose to sacrifice generality in favor of briefness.
\end{rema}

The proof of Theorem~\ref{t.conditions} will take Sections~\ref{s.bootstrap}--\ref{s.proof}.
In Section~\ref{s.open} we will prove that there exist nonempty $C^2$-open sets of
IFS's satisfying the hypotheses of the theorem (provided the number $\ell$ of generators 
is large enough, depending on the manifold $M$).
Since the measure $\nu$ produced by Theorem~\ref{t.conditions} 
satisfies precisely the conclusions of Theorem~\ref{t.mainflag},
the latter follows.
As we have seen in \S~\ref{ss.prelim_flag_dyn},
Theorem~\ref{t.mainflag} implies Theorem~\ref{t.main}.

\section{The bootstrapping procedure} \label{s.bootstrap}

As explained in the Introduction,
the measure $\nu$ in Theorem~\ref{t.conditions} will 
be obtained as the limit of a sequence of measures supported on periodic orbits,
and this sequence is constructed recursively by a ``bootstrapping procedure''.
We state below Proposition~\ref{p.improve}, which gives the recursive step of the procedure.
Then we explain how Theorem~\ref{t.conditions} follows from that proposition and Lemma~\ref{l.pretaporter}.
The proof of the proposition is given in Section~\ref{s.proof}.

\medskip

To begin, we need a few definitions.

The euclidian angle between two nonzero vectors $u$, $v \in \RR^d$ is denoted by $\angle (u,v)$.
Consider the open cone of  $\cC \subset \RR^d$ consisting of
the vectors
$$
\cC = \big\{\vec{\lambda}=(\la_1,\dots,\la_d) \in \RR^d ; \; 0>\la_1>\cdots>\la_d \big\}.
$$
We call a function $\tau \colon \cC \to \RR$ \emph{projective} if 
$\tau (t \vec{\lambda}) = \tau (\vec{\lambda})$ for all $\vec{\lambda}\in \cC$ and $t>0$.

Suppose $z = (\omega, y)$ is a periodic point of $\phi_G$ of period $p$.
Then $\omega = w^\infty$, where $w$ is the finite word $\omega_0\dots\omega_{p-1}$,
and $y\in M$ is a fixed point of $g_{[w]}$.
Let us denote 
$$
\vec{\lambda}(z) := (\la_1(z),\dots,\la_d(z)), 
$$
where 
$\la_1(z)\ge \cdots \ge \la_d(z)$ are the Lyapunov
exponents of the skew-product map $\phi_G$
with respect to the invariant measure supported on the periodic
orbit  of $z$.
If these exponents
are all different 
we define the 
\emph{stable flag} \label{def.stable_flag}
of $z$ by  $\sflag(z): =(S_1(z) \subset \cdots  \subset S_d(z))$ 
where
$$
S_i(z)=E_1(z)\oplus \cdots \oplus E_i (z)
$$
and $E_i(z)$ is the eigenspace of $Dg_{[w]}(y)$ associated to $\lambda_i(z)$.
Then $\sflag(z)$ is a hyperbolic attracting fixed point of $\cF g_{[w]}$.
If in addition all exponents $\la_i(z)$ are negative (i.e., $\vec\lambda(z) \in \cC$) then 
$(z,\sflag(z))$ is a hyperbolic attracting periodic point of $\cF \phi_G$.
(Formal proofs of these assertions will be given in Section~\ref{s.triang}.)

We assume a Riemannian metric was fixed on the manifold $M$.
This induces a Riemannian metric on the flag manifold $\cF M$, as we will see in  Section~\ref{s.triang}.

\begin{prop}[Improving a periodic orbit]\label{p.improve}
Consider a finite set $G=\{g_0,\dots , g_{\ell-1}\}$ of $\Diff^2(M)$ satisfying the hypotheses
of Theorem~\ref{t.conditions}. 
Then there exists a projective continuous function $\tau \colon \cC \to (0,1)$
such that the following holds:

Given  numbers $\theta$, $\epsilon$, $\delta>0$
and a periodic point $z$ of $\phi_G$ 
with $\vec{\lambda}(z)\in \cC$,
there exists another periodic point $\tilde z$ of larger period
with the following properties:
\begin{enumerate}
\item \label{i.p.1}
The vector $\vec{\lambda}(\tilde z)$ belongs to $\cC$
and satisfies
\begin{equation}\label{e.angle_and_norm}
0 < \|\vec{\lambda}(\tilde z)\| < \tau(\vec{\lambda}(z))\,  \|\vec{\lambda}(z)\| 
\quad \text{and} \quad
\angle (\vec{\lambda}(\tilde z),\vec{\lambda}(z))<\theta \, .
\end{equation}
\item \label{i.p.2}
There is a positive number $\kappa < \min( 1, \| \vec{\lambda}(z) \|)$
such that the orbit of $(\tilde z, \sflag (\tilde z))$ under $\cF \phi_G$
$\epsilon$-shadows the orbit of $(z,\sflag(z))$
during a proportion $1-\kappa$ of the time.
\item \label{i.p.3}
The orbit of $(\tilde z, \sflag (\tilde z))$ under $\cF \phi_G$ is 
$\delta$-dense in $\ell^\ZZ \times \cF M$.
\end{enumerate}
\end{prop}

We remark that the proposition is a multidimensional version of Lemma~3 from \cite{GIKN}.

Next we explain how this proposition allows us to recursively construct 
the desired sequence of periodic measures whose limit is the measure sought after by Theorem~\ref{t.conditions}.
The other ingredients are 
Propositions~\ref{p.Lyap_permutation} and \ref{p.Lyap_in_flag}, which allow us to pass the Lyapunov exponents to the limit,
and Lemma~\ref{l.pretaporter}, which gives the ergodicity and full support.

\begin{proof}[Proof of Theorem~\ref{t.conditions}]
By assumption (\ref{i.cond_3}), there is a periodic point $z_0 \in \ell^\ZZ \times \cF M$ of $\phi_G$
such that $\vec{\lambda}(z_0) \in \cC$.
Fix a constant $\Theta>0$ such that 
the close $\Theta$-cone around $\vec{\lambda}(z_0)$, that is
$$
\big\{ \vec{u}\in \RR^d\setminus \{ 0\}  ; \; \angle (\vec{u}, \vec{\lambda}(z_0)) \le \Theta \big\},
$$
is contained in $\cC$.
Let $\tau \colon \cC \to (0,1)$ be the continuous projective 
function produced by Proposition~\ref{p.improve},
and let $\tau_0$ be its infimum on the $\Theta$-cone around $\vec{\lambda}(z_0)$.
Then $0<\tau_0<1$.
Fix sequences $(\theta_n)$, $(\epsilon_n)$ and $(\delta_n)$
of strictly positive numbers 
such that
$$
\sum_{n=0}^\infty \theta_n   < \Theta, \qquad
\sum_{n=0}^\infty \epsilon_n < \infty, \qquad
\lim_{n\to \infty} \delta_n  = 0.
$$

We will define inductively a sequence $(z_n)$ of periodic points of $\phi_G$.
Assume that $z_n$ is already defined.
Then we apply Proposition~\ref{p.improve} 
using the numbers $\epsilon_n$, $\theta_n$ and $\delta_n$
to find another periodic point $z_{n+1}$ with the following properties:
\begin{enumerate}
\item \label{i.aa}
The vector $\vec{\lambda}(z_{n+1})$ belongs to $\cC$
and satisfies
$$
0 < \|\vec{\lambda}(z_{n+1})\|< 
\tau(\vec{\lambda}(z_n)) \, \|\vec{\lambda}(z_n)\|
\quad \text{and} \quad
\angle (\vec{\lambda}(z_{n+1}),\vec{\lambda}(z_n)) < \theta_n \, .
$$
\item
There is a positive number $\kappa_n < \min( 1, \| \vec{\lambda}(z_n) \|)$
such that the orbit of $(z_{n+1}, \sflag (z_{n+1}))$ under $\cF \phi_G$
$\epsilon_n$-shadows the orbit of $(z_n,\sflag(z_n))$
during a proportion $1-\kappa_n$ of the time.
\item 
The orbit of $(z_{n+1}, \sflag (z_{n+1}))$ under $\cF \phi_G$ is 
$\delta_n$-dense in $\ell^\ZZ \times \cF M$.
\end{enumerate}
This recursively defines the sequence $(z_n)$.

Let $\nu_n$ be the $\cF \phi_G$-invariant measure
supported on the orbit of $(z_{n},\sflag(z_{n}))$.
To complete the proof, we will show that the sequence $(\nu_n)$
converges in the weak-star topology to a measure $\nu$ with the desired properties.

Observe that 
$$\angle(\vec{\lambda}(z_n),\vec{\lambda}(z_0)) < \theta_0 + \cdots + \theta_{n-1} < \Theta.$$
In particular $\tau(\vec{\lambda}(z_n)) \ge \tau_0$ for every $n$, and therefore, by (\ref{i.aa}),
$$
\| \vec{\lambda}(z_n) \| \le \tau_0^n \| \vec{\lambda}(z_0) \|
\quad \text{and} \quad
\kappa_n < \min (1, \tau_0^n \| \vec{\lambda}(z_0) \| ) \, .
$$
The latter implies that $\prod (1-\kappa_n) >0$.
Therefore all hypotheses of Lemma~\ref{l.pretaporter}
are satisfied, and we conclude that the measure $\nu = \lim \nu_n$ exists,
is ergodic, and has support 
$$
\supp \nu=
\bigcap_{n=0}^\infty \overline{\bigcup_{m=n}^{\infty} \supp \nu_n} \,  .
$$
Since each $\supp \nu_n$ is $\delta_n$-dense, and $\delta_n \to 0$,
it follows that $\nu$ has full support.

Since
$\vec{\lambda}(z_n) \to \vec{0}$,
it follows from Proposition~\ref{p.Lyap_permutation} 
that the the sequence of Furstenberg vectors $\vec\Lambda(\nu_n)$
also converges to zero.
Since the Furstenberg vector is continuous with respect to the weak-star topology,
we have that $\vec\Lambda(\nu) = \vec{0}$.
This implies that the fibered Lyapunov exponents of $\nu$ are zero (recall Example~\ref{ex.main}),
concluding the proof of Theorem~\ref{t.conditions}.
\end{proof}

\section{Exploiting positive minimality} \label{s.tour}

The aim of this section is to prove Lemma~\ref{l.group_tour_and_go_home},
a simple but slightly technical consequence of positive minimality,
which will be used in Section~\ref{s.proof} in the proof of Proposition~\ref{p.improve}.

We begin with the following lemma:

\begin{lemm}[Go home]\label{l.go_home}
Let $H=\{h_0, \dots, h_{\ell-1}\}$ be a positively minimal set of 
homeomorphisms of a compact metric space $N$.
For every nonempty open set $U \subset N$ there exists $k_0 = k_0(U) \in \NN^*$ such that
for every $x \in N$ there exists a word $w$ of length at most $k_0$ on the alphabet $\{0,\dots,\ell-1\}$ 
such that $h_{[w]}(x) \in U$. 
\end{lemm}

\begin{proof}
Fix the set $U$.
By positive minimality, for every $x \in N$ there is a word $w(x)$ on the alphabet $\{0,\dots,\ell-1\}$
such that $h_{[w(x)]}(x) \in U$.
By continuity, there is a neighborhood $V(x)$ of $x$ such that 
$h_{[w(x)]}(V(x)) \subset U$.
By compactness, we can cover $N$ by finitely many sets $V(x_i)$.
Let $k_0$ be the maximum of the lengths of the words $w(x_i)$.
\end{proof}

For the next lemma, recall from \S~\ref{ss.prelim_IFS}
the distance on $\ell^\ZZ \times N$ and the cylinder notation.
Let us also use the following notation for segments of orbits:
$f^{[0,k]}(x) := \{x, f(x), f^2(x), \dots, f^k(x)\}$.

\begin{lemm}[Tour and go home]\label{l.tour_and_go_home}
Let $H=\{h_0, \dots, h_{\ell-1}\}$ be a positively minimal set of 
homeomorphisms of a compact metric space $N$.
For every $\delta>0$ and every nonempty open set $U \subset N$,
there exists $k_1 = k_1(\delta, U) \in \NN^*$ such that
for every $x \in N$ there exists a word $w = s_0 s_1 \dots s_{k-1}$ of length $k \le k_1$ 
on the alphabet $\{0,\dots,\ell-1\}$ 
such that:

\begin{enumerate}
\item\label{i.tour_1}
for every $\omega \in \llbracket ; s_0 s_1 \dots s_{k-1} \rrbracket$,
the segment of orbit $\phi_H^{[0,k]} (\omega,x)$ 
is $\delta$-dense in $\ell^\ZZ \times N$;

\item\label{i.tour_2}
$h_{[w]}(x) \in U$.
\end{enumerate}
\end{lemm}

\begin{proof}
Let $\delta$ and $U$ be given.
Choose a finite $(\delta/2)$-dense subset $Y \subset N$.
Let $m \in \NN$ be such that $2^{-m} \le \delta$ and 
let $W$ be the set of all words  of length $2m+1$ on the alphabet $\{0,\dots,\ell-1\}$.
Enumerate the set $Y \times W$ as $\{(y_i, w_i) \; ; 1 \le i \le r \}$.
For each $i \in \{1,\dots,r\}$, let $B_i$ be the open ball of center $y_i$ and radius $\delta/2$.
Let $w_i^-$ be the initial subword of $w_i$ of length $m$,
and let $U_i = \big(h_{[w_i^-]}\big)^{-1}(B_i)$.
Define also $U_{r+1} = U$.

Let $x \in N$ be given.
We apply Lemma~\ref{l.go_home} and find a word $w'_1$ of length at most $k_0(U_1)$
such that $h_{[w'_1]}$ sends $x$ into $U_1$.
Inductively, assuming that words $w'_1$, \dots, $w'_{i-1}$ (where $i\le r + 1$)
are already defined, we apply Lemma~\ref{l.go_home} and find a word $w_i'$ 
of length at most $k_0(U_i)$
such that $h_{[w'_i]}$ sends $h_{[w'_1 w_1 \dots w'_{i-1} w_{i-1}]}(x)$ into $U_i$.
Define
$$
w = w'_1 w_1 \dots w'_r w_r w'_{r + 1} \, .
$$
We can bound the length $k$ of $w$ independently of $x$.

Property (\ref{i.tour_2}) is obviously satisfied; let us check property (\ref{i.tour_1}).
Assume that $\omega \in \llbracket ; s_0 s_1 \dots s_{k-1} \rrbracket$, where $w=s_0 s_1 \dots s_{k-1}$.
Fix any point $(\omega^*,x^*) \in \ell^\ZZ \times N$;
we will show that a point in the segment of orbit $\phi_H^{[0,k]}(\omega,x)$ is $\delta$-close to $(\omega^*,x^*)$.
Write $\omega^*=(s_n^*)_{n\in\ZZ}$.
By the $\delta/2$-denseness of $Y$ and the definition of the set $\{(y_i,w_i) ; \; 1\le i \le r\}$, 
there exists $i$ such that
$$
d(y_i , x^*) < \delta/2 \quad\text{and}\quad
w_i = s^*_{-m} \dots s^*_{m} \, .
$$
Let $n_i$ be the length of the word 
$w'_1 w_1 \dots w'_{i-1} w_{i-1} w'_i w_i^-$;
so $m \le n_i \le k-m$.
Consider the iterate $(x_{n_i},\omega_{n_i}) = \phi_H^{n_i}(\omega,x)$.
Then we have:
\begin{itemize} 
\item $x_{n_i} \in B_i$, that is, $d(x_{n_i} , y_i) < \delta/2$, and in particular $d(x_{n_i}, x^*) < \delta$.
\item $\omega_{n_i} \in \llbracket  s_0 \dots s_{n_i-1} ; s_{n_i} \dots s_{k-1} \rrbracket$ and
in particular $d(\omega_{n_i},\omega^*) \le 2^{-m} \le \delta$,
because 
$s_{n_i-m} \dots s_{n_i+m} = w_i = s^*_{-m} \dots s^*_{m}$.
\end{itemize}
This shows that $(\omega_{n_i},x_{n_i})$ and $(\omega^*,x^*)$ are $\delta$-close, concluding the proof.
\end{proof}

The following is an immediate corollary of Lemma~\ref{l.tour_and_go_home}.

\begin{lemm}[Group tour and go home]\label{l.group_tour_and_go_home}
Let $H=\{h_0, \dots, h_{\ell-1}\}$ be a positively minimal set of 
homeomorphisms of a compact metric space $N$.
For every $\delta>0$ and every nonempty open set $U \subset N$,
there exist $\rho=\rho(\delta,U)>0$ and $k_1 = k_1(\delta, U) \in \NN^*$ such that
for every ball $B \subset N$ of radius $\rho$,
there exists a word $w = s_0 s_1 \dots s_{k-1}$ 
on the alphabet $\{0,\dots,\ell-1\}$ of length $k \le k_1$ 
such that:

\begin{enumerate}
\item 
for every $(\omega,x) \in \llbracket ; s_0 s_1 \dots s_{k-1} \rrbracket \times B$,
the segment of orbit $\phi_H^{[0,k]} (\omega,x)$ 
is $\delta$-dense in $\ell^\ZZ \times N$;

\item 
$h_{[w]}(B) \subset U$.
\end{enumerate}	
\end{lemm}
	
\begin{proof}
Use Lemma~\ref{l.tour_and_go_home} and continuity.
\end{proof}

\section{Exploiting maneuverability} \label{s.adjust}

The next lemma says that if an induced IFS on the flag bundle satisfies the
maneuverability condition, then we can select orbits whose derivatives
in the upper triangular matrix form \eqref{e.M}
have approximately prescribed diagonals.

\begin{lemm}[Products with prescribed diagonals] \label{l.prescribe}
If $G = \{ g_0,\dots,g_{\ell-1}\} \subset \Diff^1(M)$ has the maneuverability property
then there exists $c>0$ such that the following holds.
For every $\eta>0$ there is $q\in \NN^*$ such that for all
$(x,\flag)\in \cF M$ and every $(\chi_1,\dots,\chi_d)\in [-c,c]^d$
there is a word $w$ of length $q$ on the alphabet $\{0,\dots,\ell-1\}$ such that
$$
\left| \frac{1}{q} \, \log M_{i,i} (Dg_{[w]} (x) , \flag)  - \chi_i \right| < \eta
$$
for all $i=1,\dots, d$.
\end{lemm}

\begin{proof}
Take $C>0$ such that, for all $s \in \{0,\dots,\ell-1\}$ and $x\in M$,
\begin{equation}\label{e.def_C}
e^{-C} \le \|(Dg_s (x))^{-1}\|^{-1} \le  \|Dg_s (x)\| \le e^C \, .
\end{equation}

Using continuity and compactness, we can make maneuverability uniform in the following way:
there exists a constant $c>0$
such that for every $(x,\flag)\in\cF M$ 
and for every sequence of signs $t = (t_1,\dots,t_d)\in\{-1,+1\}^d$
there is $s \in \{0,\dots,\ell-1\}$ such that 
$$
t_i \, \log M_{i,i}(Dg_s(x),\flag) \ge c \quad \text{for each $i$.}
$$
Notice that (as a consequence of \eqref{e.subdeterminant}) the left hand side is at most $C$.
Given $\eta>0$, let $q := \lceil C/\eta \rceil$.

Now let $(x,\flag)\in \cF M$ and $(\chi_1,\dots,\chi_d)\in [-c,c]^d$ be given.
We inductively define the symbols $s_0$, $s_1$, \dots, $s_{q-1}$ forming the word $w$.
The idea of the proof is simple: at each step we look at diagonal obtained so far,
choose signs pointing towards the objective vector $(\chi_1,\dots,\chi_d)$, and apply uniform maneuverability
to pass to the next step.

Precisely, we choose the symbol $s_0$ so that for each $i=1,\dots,d$,
the number $\lambda_i^{(0)}:=\log M_{i,i} (Dg_{s_0} (x), \flag)$ has absolute value at least $c$ 
and has the same sign as $\chi_i$
(where we adopt the convention that the sign of $0$ is $+1$).
Assume that $s_0$, \dots, $s_{n-1}$ were already defined.
Let $(x_n, \flag_n) := (\cF g_{[s_0 s_1 \dots s_{n-1}]})(x,\flag)$.
Then we choose the symbol $s_n$ such that for each $i=1,\dots,d$, the number
$\lambda_i^{(n)} := \log M_{i,i} (Dg_{s_n} (x_n), \flag_n)$ 
has absolute value at least $c$ and has the same sign as the number
$$
\delta_i^{(n)} := n \cdot \chi_i - \log M_{i,i}(Dg_{[s_0 s_1 \dots s_{n-1}]}(x),\flag) 
               =  n \cdot \chi_i - \sum_{j=0}^{n-1} \lambda_i^{(j)} \, .
$$

Let us prove that this sequence of symbols has the required properties.
Let $i\in\{1,\ldots,d\}$ be fixed.
We will prove the following fact, which (in view of the definition of $q$) implies the lemma:
\begin{equation}\label{e.stronger}
\left| \delta_i^{(n)} \right| \le C \quad \text{for each $n\in\{1,\ldots,q\}$.}
\end{equation}
First we check the case $n=1$.
Since $\lambda^{(0)}_i$ and $\chi_i$ have the same sign, we have
$$
\left| \delta_i^{(1)} \right| = \left| \chi_i - \lambda^{(0)}_i\right| \le \left| \lambda^{(0)}_i\right| \le C,
$$
where in the last step we used \eqref{e.def_C}.
Next, assume that \eqref{e.stronger} holds for some $n<q$. 
Then
\begin{align*}
\left| \delta^{(n+1)}_i \right| 
&=   \Big| \chi_i + \delta^{(n)}_i - \lambda^{(n)}_i \Big| \\
&\le \Big| \chi_i \Big| + \Big| \delta^{(n)}_i - \lambda^{(n)}_i\Big|  \\
&=   \Big| \chi_i \Big| + \Big| \big|\delta^{(n)}_i\big| - \big|\lambda^{(n)}_i\big| \Big|  
\quad \text{(because $\delta^{(n)}_i$ and $\lambda^{(n)}_i$ have the same sign)}\\
&=   \underbrace{\Big| \chi_i \Big|}_{\le c} 
   + \underbrace{\max \left( \big|\delta^{(n)}_i\big|, \big|\lambda^{(n)}_i\big| \right)}_{\le C} 
   - \underbrace{\min \left( \big|\delta^{(n)}_i\big|, \big|\lambda^{(n)}_i\big| \right)}_{\ge c} \\
& \le C \, .
\end{align*}
This proves \eqref{e.stronger} and the lemma.
\end{proof}

\begin{rema}
Estimate \eqref{e.stronger} resembles general 
results by Shapley, Folkman, and Starr on the 
approximation of convex hulls of a sum of sets by points of the sum; 
see \cite[p.~396ff]{ArrowH}. 
\end{rema}

\section{Triangularity}\label{s.triang}

Estimating the size of products of matrices (and hence computing Lyapunov exponents)
may be a difficult business. 
The task is much simpler if the matrices happen to be upper triangular:
in that case the non-commutativity is tamed (see Proposition~\ref{p.triangular_cocycle} and  Lemma~\ref{l.off-diagonal} below).
Working in the flag bundle has the advantage of making 
all derivatives upper triangular, in a sense that will be made precise.

\subsection{Linearly induced map between flag manifolds}\label{ss.flag_derivative}

We continue the discussion from \S~\ref{ss.prelim_flag}.
Here we will give geometrical information about diffeomorphisms 
$\cF L \colon \cF E \to \cF E'$ induced by linear maps $L \colon E \to E'$.

Fix an integer $d \ge 2$.
Recall that the orthonormal group $\Ogroup(d)$ is a compact manifold of dimension $d(d-1)/2$,
whose tangent space at the identity matrix is the vector space $\so(d)$ of antisymmetric matrices.
For each $(i,j)$ with $1\le j < i \le d$,
let $X_{ij}$ be the $d \times d$ matrix such that its $(i,j)$-entry is $1$,
its $(j,i)$-entry is $-1$, and all other entries are zero.
Then $(X_{ij})_{1\le j < i \le d}$ is a basis of $\so(d)$.
For reasons that will become apparent later, we \emph{order} this basis as follows:
\begin{equation}\label{e.canonical}
\big( X_{d,1} \ ; \  X_{d-1,1} \, , X_{d,2} \ ; \  X_{d-2, 1} \, , X_{d-1,2} \, , X_{d,3} \ ; \dots ; \  
X_{2,1} \, , X_{3,2} \, , \dots , X_{d,d-1} \big).
\end{equation}
We call this the \emph{canonical basis} or \emph{canonical frame} of $\so(d)$.
Pushing forward by right translations, 
we extend this to a frame field on $\Ogroup(d)$,
called the \emph{canonical frame field}. 
We take on $\Ogroup(d)$ the Riemannian metric for which the canonical frames are orthonormal.\footnote{This metric is obviously right-invariant. Actually, it is also left-invariant. Indeed, a calculation shows that $\langle X, Y \rangle = - \trace XY/2$ for $X$, $Y\in \so(d)$, which is invariant under the adjoint action of the group, and therefore can be uniquely extended to a bi-invariant Riemannian metric. Another remark: this inner product is the Killing form divided by $-2(d-2)$ (if $d>2$).} 

Now let $E$ be a real vector space of dimension $d$,
endowed with an inner product.
Given any ordered flag $\flag_0 \in \ordF E$, 
recall that $\cO(\flag_0)$ represents the ordered orthonormal basis that represents $\flag_0$.
We define a bijection
$\iota_{\flag_0} \colon \cF E \to \Ogroup(d)$ as follows:
$\iota_{\flag_0}(\flag)$ is the matrix of the basis $\cO(\flag)$ with respect to the basis $\cO(\flag_0)$.
Notice that given another flag $\flag_1$, 
the following diagram commutes:
$$
\begin{tikzcd}
\ordF E \arrow{r}{\iota_{\flag_0}} \arrow{rd}[swap]{\iota_{\flag_1}}
&\Ogroup(d) \arrow{d}{\text{right translation by } \iota_{\flag_1}(\flag_0)} \\
&\Ogroup(d)
\end{tikzcd}
$$
Since right translations are diffeomorphisms, we can pull-back 
under any $\iota_{\flag_0}$ the differentiable structure of $\Ogroup(d)$
and obtain a well-defined differentiable structure of $\ordF E$.
(This makes more precise the explanation given in \S~\ref{ss.prelim_flag}.)
Analogously, right translations preserve the canonical frame field on $\Ogroup(d)$,
we can pull it back under any $\iota_{\flag_0}$ and obtain 
a well-defined frame field on $\ordF E$ that we call \emph{canonical}.
We endow $\ordF E$ with the Riemannian metric that makes these frames orthonormal.

Let $L \colon E \to E'$ be an isomorphism between real vector spaces of dimension $d$.
Endow $E$ and $E'$ with inner products and consider on $\ordF E$ and $\ordF E'$
the associated canonical frame fields.
For any $\flag \in \ordF E$, let $T(L,\flag)$  
denote the matrix of the derivative of the map $\ordF L \colon \ordF E \to \ordF E'$
with respect to the frames at the points $\flag$ and $(\ordF L)(\flag)$.

The following result is probably known, 
but we weren't able to find a reference:

\begin{prop}\label{p.flag_derivative}
The matrix $T(L,\flag)$ is upper triangular.
The entries in its diagonal form the list (with repetitions according to multiplicity):
\begin{equation}\label{e.diagonal_quotients}
\frac{M_{i,i}(L,\flag)}{M_{j,j}(L,\flag)} 
\quad \text{where $1 \le j < i \le d$.}
\end{equation}
\end{prop}

\begin{proof} 
Fix any $\flag \in \ordF E$ and let $\flag' = \cF L (\flag)$.
Let $\Phi = \Phi_{L,\flag}$ be the diffeomorphism that makes the following diagram commute:
$$
\begin{tikzcd}
\ordF E  \arrow{r}{\iota_{\flag}} \arrow{d}[swap]{\ordF L} &\Ogroup(d) \arrow{d}{\Phi} \\
\ordF E' \arrow{r}[swap]{\iota_{\flag'}}                   &\Ogroup(d) 
\end{tikzcd}
$$
Notice that $\Phi$ fixes the identity matrix,
and that $T(L,\flag)$ coincides with the matrix of $D\Phi(\Id)$ with respect to the 
canonical basis \eqref{e.canonical} of $\so(d)$.

Let $R = M(L, \flag)$.
If $Q \in \Ogroup(d)$ then $\Phi (Q)$
is the unique matrix $\hat Q \in \Ogroup(d)$ such that
$R Q = \hat{Q} \hat{R}$ for some upper triangular matrix $\hat{R}$ with positive diagonal entries.\footnote{Those familiar with the QR algorithm will recognize this equation; see Remark~\ref{r.QR_algo}.}

Given $X \in \so(d)$, let us compute $Y := D\Phi(\Id)(X)$.
Take a differentiable curve $Q(t)$ such that $Q(0) = \Id$ and $Q'(0)=X$.
Let $\hat{Q}(t) = \Phi(Q(t))$; so $Y = \hat{Q}'(0)$.
Write $R Q(t) = \hat{Q}(t) \hat{R}(t)$, 
where $\hat{R}(t)$ is upper triangular with positive diagonal entries.
Differentiating this relation at $t=0$, and using that $\hat{Q}(0) = \Id$ and $\hat{R}(0) = R$,
we obtain $R X = Y R + \hat{R}'(0)$.
In particular, $R X R^{-1} - Y$ is upper triangular.
It follows that $Y$ is the unique antisymmetric matrix
whose under-diagonal part coincides with the under-diagonal part of $R X R^{-1}$.
More explicitly, write 
$R = (r_{ij})$, $R^{-1}=(s_{ij})$, $X=(x_{ij})$, $Y = (y_{ij})$;
then
$$
i > j \ \Rightarrow \ 
y_{ij} = \sum_{k , \ell} r_{i k} x_{k \ell} s_{\ell j} \, .
$$
Consider the matrix $T(L,\flag)$ of $D\Phi(\Id)$ 
with respect to the canonical basis $\{X_{i,j}\}_{1\le j < i \le d}$ of $\so(d)$;
its $((i,j),(k,\ell))$-entry is $r_{ik} s_{\ell j}$,
which vanishes unless $k \ge i > j \ge \ell$.
In particular, this entry vanishes if $i-j > k-\ell$.
Under the chosen ordering \eqref{e.canonical} of the canonical basis $\{X_{i,j}\}$,
the sequence $i-j$ is nonincreasing.
Therefore $T(L,\flag)$ is an upper triangular matrix.
What are its diagonal entries?
The $((i,j),(i,j))$-entry is $r_{ii} s_{jj} = r_{ii}/r_{jj}$.
\end{proof}

Let us now consider non-oriented flags.
Let $H$ be the subgroup of $\Ogroup(d)$ formed by the matrices
\begin{equation}\label{e.pm}
\begin{pmatrix}
\pm 1 &        &       \\
      & \ddots &       \\ 
      &        & \pm 1 
\end{pmatrix}
\, .
\end{equation}
Take any $\flag_0 \in \ordF E$.
Then there is a bijection $[\iota_{\flag_0}]$ such that the following diagram commutes:
$$
\begin{tikzcd}
\ordF E \arrow{r}{\iota_{\flag_0}} \arrow{d}[swap]{\text{disorientation map}} 
&\Ogroup(d) \arrow{d}{\text{coset projection}}\\
  \cF E \arrow{r}[swap]{[\iota_{\flag_0}]}   &\Ogroup(d)/H
\end{tikzcd} 
$$
The adjoint action of $H$ does not preserve the canonical frame on $\so(d)$,
but each vector in the frame is either preserved or multiplied by $-1$,
so the action preserves an ``\emph{up-to-sign frame}''. 
We push it forward by right translations and obtain a field of up-to-sign frames
on $\Ogroup(d)/H$, which is then pulled back to a well-defined field of up-to-sign frames on $\cF E$.
There is are unique Riemannian metrics that make these up-to-sign frames orthonormal.

Now consider the diffeomorphism $\cF L \colon \cF E \to \cF E'$ 
induced by a liner isomorphism $L \colon E \to E'$.
Let $T(L,\flag)$  
denote the up-to-sign matrix of the derivative of the map $\cF L$
with respect to the up-to-sign frames at the points $\flag$ and $(\ordF L)(\flag)$.
Notice that the diagonal entries are well-defined.
It follows from Proposition~\ref{p.flag_derivative}
that this ``matrix'' is upper-triangular and that 
its diagonal entries are the numbers \eqref{e.diagonal_quotients}.

As a consequence, we have the following fact\footnote{Similar results are obtained in \cite{ShubVasquez}; see Lemma~4.}:

\begin{coro}[Stable flag]\label{c.stable} 
Suppose that $L \colon E \to E$ is a linear isomorphism 
whose eigenvalues have distinct moduli and are ordered as $|\lambda_1| > \cdots > |\lambda_d|$.
Consider the flag $\sflag = (S_i) \in \cF E$  
where $S_i$ is spanned by eigenvectors corresponding to the first $i$ eigenvalues.
Then $\sflag$ is a hyperbolic attracting fixed point of $\cF L$. 
\end{coro}

\begin{rema}\label{r.QR_algo}\footnote{We thank Carlos Tomei for telling us about the QR algorithm.} 
The \emph{QR algorithm} is the most widely used numerical method to compute 
the eigenvalues of a matrix $A_0 \in \GL(d,\RR)$ (see \cite[p.~356]{Watkins}).
It runs as follows:
starting with $n=0$, compute the QR decomposition of $A_n$, say, $A_n = Q_n R_n$,
let $A_{n+1} := R_n Q_n$, increment $n$, and repeat.
Let us interpret the sequence of matrices $A_n$ produced by the algorithm 
in terms of the diffeomorphism $\cF A_0 \colon \cF \RR^d \to \cF \RR^d$.
If $\flag_0$ is the canonical flag of $\RR^n$
then $A_n$ is the matrix of $A_0$ with respect to an orthonormal basis that represents
the flag
$\flag_n := (\cF A_0)^n(\flag_0)$.
If the eigenvalues of $A_0$ have different moduli then 
the sequence $(\flag_n)$ converges.
(Actually $\cF A_0$ is a Morse--Smale diffeomorphism whose periodic points are fixed: see \cite{ShubVasquez}.)
It follows that the sequence $(A_n)$ converges to upper triangular form.
In particular, if $n$ is large then the diagonal entries of $A_n$ 
give approximations to the eigenvalues of $A_0$.
(In practice, the algorithm is modified in order to accelerate convergence and reduce computational cost.)
\end{rema}

\subsection{Geometry of the flag bundles}\label{ss.bundle_geometry}

Now consider a compact connected manifold $M$.
We will discuss in more detail the flag bundles $\ordF M$ and $\cF M$,
defined in Example~\ref{ex.flag_derivative_cocycle}.

The tangent space of the flag manifold $\ordF M$
at a point $\xi = (x,\flag)\in \ordF M$
has a canonical subspace called the \emph{vertical subspace},
denoted by $\mathrm{Vert}(\xi)$, which is the tangent space 
of the fiber $(\ordF M)_x = \ordF(T_x M)$ at $\xi$.

Now fix a Riemannian metric on $M$.
Then we can define the \emph{horizontal subspace}, denoted $\mathrm{Horiz}(\xi)$, as follows:
for each smooth curve starting at the point $x$, consider 
the parallel transport of the flag $\flag$,
which gives a smooth curve in the manifold $\ordF M$;
consider the initial velocity $w \in T_\xi(\ordF M)$ of the curve.
Then $\mathrm{Horiz}(\xi)$ consists of all vectors $w$ obtained in this form.\footnote{This field of horizontal subspaces is actually an Ehresmann connection on the principal bundle $\ordF M$.} 
The tangent space of $\ordF M$ at $\xi$ splits as
\begin{equation}\label{e.vert_horiz}
T_\xi(\ordF M) = \mathrm{Vert}(\xi) \oplus \mathrm{Horiz}(\xi) \, .
\end{equation}
If $\pi \colon \ordF M \to M$ is the projection, then 
$\mathrm{Vert}(\xi)$ is the kernel of the derivative $D \pi(\xi)$,
and the restriction of to $\mathrm{Horiz}(\xi)$ is an isomorphism onto $T_x M$.

Since $\mathrm{Vert}(\xi) = T_\flag (\ordF(T_x M))$, there is a well-defined canonical frame on $\mathrm{Vert(\xi)}$ (see the previous subsection).
On the other hand, there is a natural frame on $\mathrm{Horiz}(\xi)$, namely, 
the unique frame sent by $D \pi(\xi)$ to the frame $\cO(\flag)$ of $T_x M$.
By concatenating these two frames (in the same order as in \eqref{e.vert_horiz}),
we obtain a frame of $T_\xi(\ordF M)$, which will called \emph{canonical}.
So we have defined a canonical field of frames on $\ordF M$.
We endow $\ordF M$ with the Riemannian metric that makes these frames orthonormal.

Suppose that $g \colon M \to M$ is a $C^2$ diffeomorphism.
Then $g$ induces a $C^1$ diffeomorphism $\ordF g \colon \ordF M \to \ordF M$.
Take $\xi = (x,\flag) \in \ordF M$ and consider the derivative 
$D(\ordF g)(\xi)$;
expressing it as a matrix with respect to the canonical frames,
we obtain:
\renewcommand{\arraystretch}{1.5} 
\begin{equation}\label{e.big_triangular}
\left(
\begin{array}{c|c}
T(Dg(x),\flag) & *                \\
\hline 
0              & M(Dg(x),\flag)
\end{array}
\right),
\end{equation}
where $M$ and $T$ are the matrices defined in \S\S~\ref{ss.prelim_flag} and \ref{ss.flag_derivative},
respectively.
Recalling Proposition~\ref{p.flag_derivative}, we see that the matrix \eqref{e.big_triangular}
is upper triangular, and the entries in its diagonal are:
\begin{equation}\label{e.big_diagonal}
\frac{M_{i,i}(Dg(x),\flag)}{M_{j,j}(Dg(x),\flag)} 
\quad \text{and} \quad
M_{k,k}(Dg(x),\flag) \quad \text{(where $1 \le j < i \le d$ and $1 \le k \le d$).}
\end{equation}

\renewcommand{\arraystretch}{1} 

The case of non-oriented flags is analogous.
There is a canonical field of ``up-to-sign frames'' on $\cF M$.
Given a diffeomorphism $g$ and a point $\xi = (x,\flag) \in \cF M$ 
the entries of the matrix \eqref{e.big_triangular} which represents $D(\cF g)(\xi)$;
are defined up to sign,
while the diagonal entries are well-defined.

\begin{rema}
The triangularity property seem above can be summarized abstractly as follows:
The fiber bundle $\ordF (\ordF M) \to \ordF M$ has a special section 
which is invariant for under $\ordF (\ordF g)$, for any $g \in \Diff^2(M)$. 
\end{rema}

\subsection{The Lyapunov exponents in the flag bundle} \label{ss.lyapunovflag}

Using the description obtained in \S~\ref{ss.flag_derivative} 
for the derivatives of linearly-induced maps on flag manifolds,
the proof of Proposition~\ref{p.Lyap_in_flag} will be reduced to the following standard result:

\begin{prop}[Lyapunov exponents of triangular cocycles]\label{p.triangular_cocycle}
Let $T \colon X \to X$ be a continuous transformation of a compact metric space,
and let $\mu$ be an ergodic invariant measure for $T$.
Suppose $A \colon X \to \GL(d,\RR)$ is a continuous map
that takes values on upper-triangular matrices.
Consider the morphism $S(x,v) = (T(x),A(x)\cdot v)$ of the trivial vector bundle $V = X \times \RR^d$.
Then the Lyapunov exponents of $S$ with respect to $\mu$,
repeated according to multiplicity, are the numbers
$$
\int_X \log |a_{ii}(x)| \, d\mu(x), \quad i=1, 2, \dots, d,
$$
where $a_{ii}(x)$ are the diagonal entries of the matrix $A(x)$.
\end{prop}

\begin{proof}
See Lemma~6.2 in \cite{JPS} (as mentioned there, the argument applies to any triangular cocycle 
over a compact metric space).\footnote{As we mentioned in Remark~\ref{r.oseledets_trick},
Oseledets \cite{Oseledets} reduced the proof of his theorem to the triangular case.}
\end{proof}

\begin{proof}[Proof of Proposition~\ref{p.Lyap_in_flag}]
We fix a continuous map $T \colon X \to X$ of a compact metric space $X$,
a vector bundle $V$ of rank $d$ over $X$ endowed with a Riemannian metric,
and a vector bundle morphism $S \colon V \to V$ over $T$.
that is invertible in each fiber.

First we will prove the corresponding statement of Proposition~\ref{p.Lyap_in_flag} for oriented flags.
Let $\pi \colon \ordF V \to X$ be the bundle projection.
Consider the vector bundle $W$ over $\ordF V$ whose fiber $W_\xi$ over $\xi \in \cF V$ 
is the tangent space of the flag manifold $\ordF(V_{\pi(\xi)})$ at $\xi$.
Using the canonical frame field explained in \S~\ref{ss.flag_derivative},
this vector bundle can be trivialized as $W = (\ordF V) \times \RR^{d(d-1)/2}$.
The derivative of $\ordF S \colon \ordF V \to \ordF V$ induces a vector bundle automorphism
$U \colon W \to W$, which under the trivializing coordinates has a generator 
$A \colon X \to \GL({d(d-1)/2},\RR)$ taking values on upper triangular matrices,
with diagonals given by expressions as \eqref{e.diagonal_quotients}.
Applying Proposition~\ref{p.triangular_cocycle}, the desired result follows. 

The case of non-oriented flags follows easily from the oriented case and the following observation:
any $\cF S$-invariant ergodic probability on $\cF V$ can be lifted to 
a $\ordF S$-invariant ergodic probability on $\ordF V$.
\end{proof}

\subsection{Products of triangular matrices} \label{ss.triprod}

Proposition~\ref{p.triangular_cocycle}
indicates that off-diagonal entries of
``random'' products of triangular matrices are dominated by the diagonals.
We will also need the following simple deterministic version of this fact:

\begin{lemm}[Products of triangular matrices]\label{l.off-diagonal}
Given numbers $d \in \NN^*$, $C>1$, $\lambda \in \RR$, and $\eta>0$,
there exists $N \in \NN^*$ with the following property: 
If $R(0)$, $R(1)$, \dots is a sequence of 
upper triangular matrices whose entries satisfy the bounds
\begin{align*}
|R_{i,j}(n)| &\le C,  \\
C^{-1} \le |R_{i,i} (n)| &\le e^\lambda.
\end{align*}
Then, for every $n \ge N$, 
$$
\| R(n-1) \cdots R(0) \| \le e^{(\lambda + \eta)n} \, .
$$
\end{lemm}

\begin{proof}
Let $P(n) = R(n-1) \cdots R(0)$.
For $i\le j$, the $(i,j)$-entry of this matrix is given by
$$
R_{i,j}(n) = \sum R_{j_n, j_{n-1}}(n-1) \cdots R_{j_2,j_1}(1) R_{j_1,j_0}(0) ,
$$
where the sum is taken over all non-increasing sequences 
\begin{equation}\label{e.non-increasing}
j = j_0 \ge j_1 \ge \dots \ge j_n = i  \, .
\end{equation}
Fix one of those sequences. 
Let $m$ be the number of strict inequalities that appear in \eqref{e.non-increasing}.
Then 
$$
|R_{j_n, j_{n-1}}(n-1) \cdots R_{j_2,j_1}(1) R_{j_1,j_0}(0)| \le 
C^m e^{\lambda (n-m)} \le
C^{2m} e^{\lambda n} \, .
$$
Using that $m \le d-1$, and summing over all sequences \eqref{e.non-increasing}, we obtain
$$
| P_{i,j}(n) | \le C^{2(d-1)} \binom{n}{j-i} e^{\lambda n} \, .
$$
Since these binomial coefficients are polynomial functions of $n$,
the lemma follows.
\end{proof}

\section{Improving a periodic orbit} \label{s.proof}

In this section we prove Proposition~\ref{p.improve}.
A very rough outline of the proof can be found in \S~\ref{ss.boot};
the basic notation can be found in Figure~\ref{f.strategy}.
Let us give some extra informal explanations before the actual proof:
\begin{itemize}
\item 
The matrices in \eqref{e.big_triangular} 
that represent the derivatives are upper-triangular;
to determine the Lyapunov exponents of the new periodic orbit we only need to know the matrix diagonals.
In particular, since the ``tour and go home'' proportion of the orbit 
will be much smaller than the rest, it will be negligible for the estimation of Lyapunov exponents.

\item
Nevertheless, we still will need to estimate norms of derivatives,
since we will want to fit images of balls inside balls (recall \S~\ref{ss.boot}).
Lemma \ref{l.off-diagonal} allows us to basically disregard the off-diagonal elements.
It is important to apply Lemma~\ref{l.off-diagonal} only after multiplying together 
the derivatives along each segment of orbit provided by Lemma~\ref{l.prescribe}, 
because then the diagonal is controlled.

\item
Let $\vec{\lambda}=(\lambda_i)$ be the Lyapunov vector of the given periodic orbit.
Let $(\tilde \lambda_i)$ denote the (still to be determined) new exponents,
and let $(\chi_i)$ be the exponents along the correcting phase 
(corresponding to $h_2$ in \S~\ref{ss.boot}).
So $\tilde \lambda_i \simeq (1-\kappa_0) \lambda_i + \kappa_0 \chi_i$,
where $\kappa_0$ is the (still to be determined) approximate proportion of the correcting phase.
We want the vectors $(\tilde \lambda_i)$ and $(\lambda_i)$ to form a small angle;
so we take $\chi_i = -a \lambda_i$, for some proportionality factor $a>0$.
The largest correcting exponent we can take is the number $c$ given by Lemma~\ref{l.prescribe}.
We take $\chi_d = c$, and so we determine $a = c/|\lambda_d|$.

\item 
Let $\gamma = \gamma(\vec\lambda)$ be the least gap in the sequence $0 > \lambda_1 > \cdots > \lambda_d$,
that is,
\begin{equation}\label{e.gamma}
\gamma(\vec\lambda) := 
\min \big\{ -\lambda_1, \lambda_1-\lambda_2, \lambda_2-\lambda_3, \dots, \lambda_{d-1}-\lambda_d \big\} \, .
\end{equation}
It follows from the description of derivatives from \eqref{ss.bundle_geometry}
that the maximum expansion exponent (on $\cF M$) around the original orbit is $-\gamma(\vec \lambda)$.
Analogously, the maximum expansion exponent along the correcting phase is $\chi_d = c$.
Since we want the ball $B_2$ to have (much) smaller radius than $B_0$,
it is necessary that $(1-\kappa_0)(-\gamma) + \kappa_0 c < 0$,
that is, $\kappa_0 < \gamma/(\gamma+c)$. 
We choose 
$$
\kappa_0 := \frac{\gamma}{2C} \, , 
$$
where $C > \max(\gamma,c)$ is an upper bound for all expansions.

\item
Finally, we estimate the factor
$$
\frac{|\tilde \lambda_i|}{|\lambda_i|} \simeq (1-\kappa_0) + \kappa_0 (-a) \le 1 - a \kappa_0 
= 1 - \frac{c\gamma}{2C |\lambda_d|} \, .
$$
We choose $\tau$ as something bigger than the right-hand side, e.g.:
\begin{equation}\label{e.tau}
\tau(\vec{\lambda}) := 1 -  \frac{c}{C} \cdot 
\frac{\gamma(\vec\lambda)}{|\lambda_d|} \, ,
\end{equation}
This is a continuous projective function, as required.
\end{itemize}

Now we give the formal proof:

\begin{proof}[Proof of Proposition~\ref{p.improve}]
Suppose the set $G = \{g_0, \dots, g_{\ell-1}\} \subset \Diff^2(M)$
satisfies the assumptions (\ref{i.cond_1}), (\ref{i.cond_2}) and (\ref{i.cond_3}) 
of Theorem~\ref{t.conditions}.
Fix constants $C>c>0$, where $C$ satisfies \eqref{e.def_C} 
and $c$ is given by Lemma~\ref{l.prescribe}.
For any $\vec{\lambda} \in \cC$, let $\gamma(\vec{\lambda})$ be the ``gap'' defined by \eqref{e.gamma}.
Define the  function $\tau \colon \cC \to (0,1)$ by \eqref{e.tau}.

Now fix a periodic point $z$ of $\phi_G$ with $\vec{\lambda}(z) \in \cC$
and constants $\theta$, $\epsilon$, $\delta>0$.
For simplicity, write $(\lambda_1,\dots,\lambda_d) = \vec \lambda = \vec \lambda(z)$ and
$\gamma = \gamma(\vec{\lambda})$. 
Let $p$ be the minimal period of $z$;
so $z = (w^\infty, x^0)$ where $w$ is a word of length $p$
and $x^0 \in M$.
Let $\flag^0 = \sflag(z)$ and 
$\xi_0 = (x^0, \flag^0) \in \cF M$.
By definition,
\begin{equation}\label{e.lambda_obvious}
\lambda_i = \frac{1}{p} \log M_{i,i} \big( Dg_{[w]}(x^0) , \flag^0\big) \, .
\end{equation}

Let $\eta>0$ be a very small number; we will see along the proof how small it needs to be.
Of course, each smallness condition that will appear must involve only the objects that defined up to this point. 

Let $n_0 \in \NN^*$ be such that 
\begin{equation}\label{e.n_0}
2^{-p n_0} < \epsilon.
\end{equation}

To simplify notation, let $h_s = \cF g_s$ for each letter $s=0,\dots,\ell-1$,
and so $h_{[w]} = \cF g_{[w]}$ for each word on this alphabet.

If $\xi \in \cF M$ and $r>0$, let
$B(\xi, r) \subset \cF M$ denote the ball of center $\xi$ and radius 
$r$, with respect to the Riemannian norm on $\cF M$ explained in \S~\ref{ss.bundle_geometry}.

\begin{clai}\label{cl.norm_1st}
There exist $\rho>0$ and $n_1 \in \NN^*$
such that: 
\begin{equation}\label{e.norm_1st}
\xi \in B(\xi_0, \rho), \ j \ge n_1 \ \Rightarrow \ 
\big\| D h_{[w^j]} (\xi) \big\| \le \exp \big[( - \gamma + 3\eta) pj \big] \, .
\end{equation}
In particular,
\begin{equation}\label{e.invariance}
h_{[w^{n_1}]} \big( B (\xi_0,\rho) \big) \subset B(\xi_0,\rho) \, .
\end{equation}
\end{clai}

\begin{proof}[Proof of the claim]
The map $h_{[w]}$ has $\xi_0$ as an attracting fixed point.
Recalling \eqref{e.big_diagonal}, we see that
the moduli of the eigenvalues of $Dh_{[w]}(\xi_0)$ are the numbers
$$
e^{\lambda_k p } \quad \text{and} \quad
e^{(\lambda_i - \lambda_j)p}
\quad \text{where $1 \le k \le d$ and $1 \le j < i \le d$.}
$$
In particular, the spectral radius of $Dh_{[w]}(\xi_0)$ is $e^{- \gamma p}$.
Therefore there is $j_0 \in \NN^*$ such that 
$$
\|Dh_{[w^{j_0}]}(\xi_0)\| \le \exp\big[ (-\gamma +\eta) p j_0 \big].
$$
By continuity of $Dh_{[w]}$, there is $\rho >0$ such that
$$
\xi \in B(\xi_0,\rho) \ \Rightarrow \ 
\big\|Dh_{[w^{j_0}]}(\xi)\big\| \le \exp\big[ (-\gamma + 2\eta) p j_0 \big].
$$
The right-hand side is less than $1$ (because $\eta$ is very small);
in particular, $h_{[w^{j_0}]}$ maps $B(\xi_0,\rho)$ into itself.
Let $C_1 := \sup_{\cF M} \|Dh_{[w]}\| \ge 1$.
It follows from submultiplicativity of norms that for every $j\ge j_0$ and $\xi \in B$, 
we have
\begin{align*}
\big\|Dh_{[w^j]} (\xi)\big\| 
&\le C_1^{j - j_0 \lfloor j/j_0 \rfloor} \exp \big[ (-\gamma + 2\eta) p j_0 \lfloor j/j_0 \rfloor \big] \\
&\le C_1^{j_0-1} \exp \big[ (-\gamma + 2\eta) p (j-j_0+1) \big] .
\end{align*}
For sufficiently large $j$, the right-hand side is less than $\exp\big[ (-\gamma + 3\eta) p j \big]$,
as we wanted to show.
\end{proof}

By continuity, we can reduce $\rho$ so that it has the following
additional properties:
\begin{equation}\label{e.drift_epsilon} 
h_{[\hat{w}]} (B(\xi_0,\rho)) \subset B\big(h_{[\hat{w}]}(\xi_0),\epsilon\big) \quad 
\text{for any word $\hat{w}$ of length $\le p n_1$.}
\end{equation}
and
\begin{equation}\label{e.diag_1st}
e^{-\eta} < \frac{M_{ii}\left(Dg_{[w]}(x),\flag \right)}{M_{ii}\left(Dg_{[w]}(x^0), \flag^0 \right)} < e^\eta \quad \forall (x,\flag) \in B(\xi_0, \rho), \ \forall i=1,\dots, d. 
\end{equation}

Define
\begin{equation}\label{e.chi}
\chi_i := \frac{c |\lambda_i|}{|\lambda_d|} \, .
\end{equation}
Notice that the vectors $(\chi_i)$ and $(\lambda_i)$ are collinear, and that
\begin{equation}\label{e.ordered_and_bounded}
0 < \chi_1 < \chi_2 < \cdots < \chi_d = c \, .
\end{equation}

Using Lemma~\ref{l.prescribe}, we find 
$q^* \in \NN^*$ associated to the precision $\eta$.
We inductively define a sequence 
$w^*_j$ of words of length $q^*$ 
as follows:
Assume $w^*_0$, \dots, $w^*_{j-1}$
were already defined.
Let $\xi_j^* = (x_j^*,\flag_j^*) =  h_{[w_{j-1}^* \cdots w_0^*]}(\xi_0)$.
Since $|\chi_i| \le c$, we can apply Lemma~\ref{l.prescribe} 
and select a word $w^*_j$ of length $q^*$ such that
\begin{equation}\label{e.adjust_concl}
\left| \frac{1}{q^*} \, \log M_{i,i} \left(Dg_{[w^*_j]} (x_j^*) , \flag_j^*\right) - \chi_i  \right| < \eta
\quad \text{for all $i=1,\dots, d$.}
\end{equation}

For any $j$, 
consider the (up-to-sign) matrix of $Dh_{[w_j]}(\xi^*_j)$ with respect to the canonical frames,
as explained in \S~\ref{ss.bundle_geometry}.
The diagonal entries are well-defined.
By \eqref{e.big_triangular}, 
and using \eqref{e.adjust_concl} and \eqref{e.ordered_and_bounded},
we see that all these diagonal entries are less than $e^{(c + 2\eta)q}$.
We apply Lemma~\ref{l.off-diagonal} and find $N \in \NN^*$ 
such that 
\begin{equation}\label{e.norm_block_0}
\big\| Dh_{[w^*_j w^*_{j+1} \dots w^*_{j+N-1}]}(\xi^*_j) \big\| \le 
\exp\left[\big( c + 3\eta \big) q^*N  \right] \quad \text{for each $j\ge 0$}\, .
\end{equation}

Let us concatenate the words $w^*_0$, $w^*_1$, \dots \  in blocks of $N$ words, thus forming a sequence
$w_0$, $w_1$, \dots \ of words of length $q := q^* N$:
$$
\underbrace{w^*_0 \dots w^*_{N-1}}_{w_0} \underbrace{w^*_N \dots w^*_{2N-1}}_{w_1} \underbrace{w^*_{2N} \dots w^*_{3N-1}}_{w_2}\dots
$$
So \eqref{e.adjust_concl} gives
\begin{equation}\label{e.adjust_2nd}
\left| \frac{1}{q} \, \log M_{i,i} \left(Dg_{[w_{Nj}]} (x_{Nj}^*) , \flag_{Nj}^*\right) - \chi_i  \right| < \eta
\quad \text{for all $i=1,\dots, d$.}
\end{equation}
Also, defining $\xi_j := \xi^*_{Nj}$,
\eqref{e.norm_block_0} gives
\begin{equation}\label{e.norm_2nd}
\big\| Dh_{[w_j]}(\xi_j) \big\| \le 
\exp\left[\big( c + 3\eta \big) q \right] \quad \text{for each $j\ge 0$} \, .
\end{equation}

We apply Lemma~\ref{l.group_tour_and_go_home} to the set $U = B(\xi_0,\rho)$
and find numbers $\rho'>0$ and $k_1\in \NN$ with the following properties:
For every ball $B' \subset \cF M$ of radius $\rho'$,
there exists a word $w' = s_0 s_1 \dots s_{k-1}$ 
of length $k \le k_1$ 
such that:
\begin{equation} \label{e.group_tour}
\left.
\begin{array}{l}
\text{for every $(\omega,x) \in \llbracket ; s_0 s_1 \dots s_{k-1} \rrbracket \times B'$,} \\
\text{the segment of orbit $\phi_H^{[0,k]} (\omega,x)$ is $\delta$-dense in $\ell^\ZZ \times \cF M$}
\end{array}
\right\}
\end{equation}
and
\begin{equation}\label{e.go_home}
h_{[w']}(B') \subset B(\xi_0,\rho).
\end{equation}

Using that the maps in $G$ are $C^2$,
we reduce $\rho'$ if necessary so that it has the following additional property:
if $\xi'=(x',\flag')$ and $\xi''=(x'',\flag'') \in \cF M$ are $\rho'$-close 
then for any word $\hat{w}$ of length $q$, 
\begin{align}
e^{-\eta} &< 
\frac{\big\| D h_{[\hat{w}]} (\xi') \big\|}{\big\| D h_{[\hat{w}]} (\xi'') \big\|}
< e^{\eta} \, , \label{e.unif_cont_norms} \\
e^{-\eta} &< \frac{M_{ii}\left(Dg_{[\hat{w}]}(x'),\flag' \right)}{M_{ii}\left(Dg_{[\hat{w}]}(x''), \flag'' \right)} < e^\eta 
\text{ for each $i=1,\dots,d$.} \label{e.diag_2nd}
\end{align}

Define
\begin{align}
\kappa_0 &:= \frac{\gamma}{2C} \, ,   \label{e.kappa}  \\
-\beta   &:= (1-\kappa_0)(-\gamma) + \kappa_0 c   \, . \label{e.beta}
\end{align}
Since $c$ and $\gamma$ are less than $C$, we have
$$
\beta =  \left(1 - \frac{\gamma+c}{2C} \right) \gamma > 0.
$$

Now we choose large integers $n$, $m$ with the following properties:
\begin{gather}
n \ge n_1                                                \, , \label{e.cond_n_1}\\
\kappa_0 - \eta < \frac{qm}{pn+qm} < \kappa_0            \, , \label{e.proportion}\\
\frac{\max(p,n_0,k_1)}{pn+qm} < \eta                     \, , \label{e.short}\\
\rho \exp\left[ -\frac{\beta}{2}(pn+qm) \right]  < \rho' \, . \label{e.radius} 
\end{gather}
Let 
$$
r_j := \rho \exp \left[ (-\gamma + 3\eta)pn \right] \exp \left[ ( c + 4\eta ) q j \right],
\quad \text{for $j=0,1,\dots,m$.}
$$
It follows from \eqref{e.cond_n_1} and \eqref{e.norm_1st} that
\begin{equation}\label{e.nested_first}
h_{[w^n]}\big( B(\xi_0, \rho) \big) \subset B(\xi_0, r_0).
\end{equation}

\begin{clai}\label{cl.radii_bound}
$r_j < \rho'$ for each $j=0,1,\dots,m$.
\end{clai}
Before proving this claim, notice 
that together with \eqref{e.norm_2nd} and \eqref{e.unif_cont_norms}, it implies that
\begin{equation}\label{e.nested}
h_{[w_j]} \big( B(\xi_j, r_j) \big) 
\subset  B \big( h_{[w_j]}(\xi_{j}), e^{( c + 4\eta ) q } r_j \big) 
=  B(\xi_{j+1}, r_{j+1}) \, . 
\end{equation}

\begin{proof}[Proof of the Claim~\ref{cl.radii_bound}]
Since $r_0 \le r_1 \le \cdots \le r_m$, we only need to estimate $r_m$.
It follows from the second inequality in \eqref{e.proportion} that
$$
r_m \le \rho\exp\Big[ \big( \underbrace{(1-\kappa_0)(-\gamma) + \kappa_0 c}_{-\beta} + 8\eta \big) (pn+qm) \big] \, . 
$$
So, imposing $\eta < \beta/14$ (which is allowed), the claim follows from \eqref{e.radius}.
\end{proof}

Let $B':=B(\xi_m,\rho')$ and find a corresponding word $w' = s_0 s_1 \dots s_{k-1}$ 
of length $k \le k_1$ 
with the properties ``group tour'' \eqref{e.group_tour} and ``go home'' \eqref{e.go_home}.

Consider the following word of length $pn + qm + k$:
\begin{equation*}
\tilde w = w^n w_0 \cdots w_{m-1} w' \, .
\end{equation*}
It follows from properties \eqref{e.nested_first}, \eqref{e.nested}, and \eqref{e.go_home} that
$$
h_{[\tilde w]} \big( B(\xi_0,\rho) \big) \subset B(\xi_0,\rho) .
$$
So $h_{[\tilde w]}$ has a fixed point $\tilde \xi = (\tilde x, \tilde \flag)$
inside $B(\xi_0,\rho)$. 
In particular, $\tilde z := (\tilde w^\infty, \tilde x)$ is a periodic point for $\phi_G$,
with period $pn + qm + k$.

This concludes the construction of the ``improved'' periodic orbit.
The rest of the proof consists of checking 
that this orbit has the desired properties (\ref{i.p.1}), (\ref{i.p.2}),  (\ref{i.p.3}).

\medskip\noindent\textit{Verifying property (\ref{i.p.1}).} 
Consider the (periodic) orbit of $(\tilde z, \tilde \flag)$ under $\cF \phi_G$:
$$
(\sigma^j(\tilde w^\infty), \tilde x_j, \tilde \flag_j) := (\cF\phi_G)^j(\tilde w^\infty, \tilde x, \tilde \flag), \quad \text{where $j \in \ZZ$.}
$$
We must estimate the Lyapunov vector 
$\vec \lambda(\tilde z) = (\tilde \lambda_1, \dots, \tilde \lambda_d)$.
For each $i=1,\dots,d$, we have
$$
\tilde \lambda_i = \frac{\log M_{ii}(Dg_{[\tilde w]}(\tilde x) , \tilde \flag)}{pn+qm+k} =
\frac{1}{pn+qm+k}\left[\sum_{j=0}^{n-1}(\mathrm{I})_j + \sum_{j=0}^{m-1}(\mathrm{II})_j + (\mathrm{III})\right],
$$
where
\begin{align*}
(\mathrm{I})_j &= 
\log M_{ii} \big( Dg_{[w]}(\tilde x_{jn}) , \tilde \flag_{jn} \big) \, , \\
(\mathrm{II})_j &=
\log M_{ii} \big(Dg_{[w_j]}(\tilde x_{pn+qj}) , \tilde \flag_{pn+qj}\big) \, , \\
(\mathrm{III}) &= \log M_{ii}\big(Dg_{[w']}(\tilde x_{pn+qm}) , \tilde \flag_{pn+qm}\big) \, .
\end{align*}
We have
\begin{alignat*}{2}
\left| \frac{(\mathrm{I})_j}{p} - \lambda_i \right| &\le \eta 
&\qquad&\text{(by \eqref{e.lambda_obvious} and \eqref{e.diag_1st}),} \\ 
\left| \frac{(\mathrm{II})_j}{q} - \chi_i \right| &\le 4\eta 
&\qquad&\text{(by \eqref{e.adjust_2nd} and \eqref{e.diag_2nd}),} \\ 
\left| \frac{(\mathrm{III})}{k}  \right| &\le C
&\qquad&\text{(by \eqref{e.def_C}).} 
\end{alignat*}
Using these estimates together with \eqref{e.proportion}, \eqref{e.short},
we obtain
\begin{align}
\tilde \lambda_i
&= 
\underbrace{\frac{pn}{pn+qm+k}}_{1 - \kappa_0 + O(\eta)} 
\underbrace{\frac{\sum (\mathrm{I})_j}{pn}}_{\lambda_i + O(\eta)} +
\underbrace{\frac{qm}{pn+qm+k}}_{\kappa_0 + O(\eta)} 
\underbrace{\frac{\sum (\mathrm{II})_j}{qm}}_{\chi_i + O(\eta)}  +
\underbrace{\frac{k}{pn+qm+k}}_{O(\eta)} 
\underbrace{\frac{(\mathrm{III})}{k}}_{O(1)}  \notag \\
&= (1-\kappa_0)\lambda_i + \kappa_0\chi_i + O(\eta). \label{e.tilde_lambda}
\end{align}
We have
\begin{multline*}
\tilde\lambda_{i-1} - \tilde\lambda_i 
=   (1-\kappa_0)(\lambda_{i-1}-\lambda_i) + \kappa_0(\chi_{i-1}-\chi_i) + O(\eta) \\
\ge (1-\kappa_0)\gamma - \kappa_0 c + O(\eta)
\ge \beta + O(\eta) > \beta/2 > 0,
\end{multline*}
taking $\eta$ small enough.
A similar calculation gives $-\tilde\lambda_1 > \beta/2 >0$.
In particular, $(\tilde \lambda_i)$ belongs to the cone $\cC$, as required.
Moreover, using \eqref{e.kappa},
$$
(1-\kappa_0)\lambda_i + \kappa_0\chi_i > \lambda_i + \kappa_0 \chi_i =  
\left[ 1 - \frac{c}{2C} \cdot \frac{\gamma}{|\lambda_d|} \right] \lambda_i \, .
$$
The quantity between square brackets is positive and strictly less than
the number $\tau(\vec\lambda)$ defined in \eqref{e.tau};
therefore, taking $\eta$ small, \eqref{e.tilde_lambda} guarantees that
$$
|\tilde \lambda_i| < \tau(\vec \lambda) |\lambda_i| \, .
$$
This gives the desired norm inequality in \eqref{e.angle_and_norm}.
Notice that the vector $(\tilde \lambda_i)$ is not too close to zero; indeed, 
$|\tilde \lambda_i| \ge \beta/2$.
So, taking $\eta$ small enough, 
the angle inequality in \eqref{e.angle_and_norm} follows from \eqref{e.tilde_lambda}
and the fact that the vectors $(\chi_i)$ and $(\lambda_i)$ are collinear.
We have checked part~(\ref{i.p.1}) of the proposition.

\medskip\noindent\textit{Verifying property (\ref{i.p.2}).} 
Write the original periodic orbit 
in $\ell^\ZZ \times \cF M$ as 
$$
(\sigma^j(w^\infty), x_j^0, \flag_j^0) := (\cF\phi_G)^j(w^\infty, x^0, \flag^0), \quad \text{where $j \in \ZZ$.}
$$
Since $(\tilde x_0, \tilde \flag_0)$ is inside the ball $B(\xi_0,\rho)$,
by the invariance condition \eqref{e.invariance} 
the points $(\tilde x_{jpn_1}, \tilde \flag_{jpn_1})$ with $j=0,1,\dots,n/n_1$
are also inside the ball.
So, it follows from \eqref{e.drift_epsilon} that
$$
d \big( (\tilde x_j, \tilde \flag_j),  (\tilde x^0_j, \tilde \flag^0_j) \big) < \epsilon 
\quad \text{for all $j$ with $0 \le j \le pn$.}
$$
So it follows from \eqref{e.n_0} that 
$$
d \big( (\sigma^j(\tilde w^\infty), \tilde x_j, \tilde \flag_j),  (\sigma^j(w^\infty), \tilde x^0_j, \tilde \flag^0_j) \big) < \epsilon 
\quad \text{for all $j$ with $n_0 \le j \le pn - n_0$.}
$$
Therefore the orbit of $(\tilde w^\infty, \tilde x, \tilde \flag)$ $\epsilon$-shadows 
the orbit of $(w^\infty, x^0, \flag^0)$ during a proportion 
$$
\frac{pn-2n_0-p}{pn+qm+k}
$$
of the time.
It follows from \eqref{e.proportion} and \eqref{e.short}
(taking $\eta$ small) that this proportion is greater than $1-2\kappa_0$.
So let $\kappa := 2\kappa_0 = \gamma/C$.
Notice that $\gamma < |\lambda_d| \le \| \vec \lambda \|$, so $\kappa < \max(1,\|\vec \lambda\|)$,
as required.
We have checked part~(\ref{i.p.2}) of the proposition.

\medskip\noindent\textit{Verifying property (\ref{i.p.3}).} 
Since for $j = np+qm$, the point 
$(\sigma^j(\tilde w^\infty), \tilde x_j, \tilde \flag_j)$
belongs to $\llbracket ; s_0 s_1 \dots s_{k-1} \rrbracket \times B'$;
therefore property \eqref{e.group_tour} assures that the first $k$ iterates of this point 
form a $\delta$-dense subset of $\ell^\ZZ \times \cF M$.
We have checked the last part of Proposition~\ref{p.improve}.
The proof is completed.
\end{proof}

\begin{rema}
As already mentioned in Remark~\ref{r.C2}, our proof uses $C^2$ regularity;
the precise places where we need it are \eqref{e.norm_1st} and \eqref{e.unif_cont_norms}.
(Apart from that, we will need $C^2$ regularity again in Section~\ref{s.open}.)
\end{rema}

\section{Construction of the open set of IFS's}\label{s.open}

The aim of this section is to show the existence of iterated function systems 
that $C^2$-robustly satisfy the hypotheses of Theorem~\ref{t.conditions}, i.e., to prove the following:

\begin{prop}\label{p.open_conditions}
Let $M$ be a compact connected manifold. 
There is $\ell \in \NN^*$ and a nonempty open set $\cG_0 \subset (\Diff^2(M))^\ell$ such that 
every $G\in\cG_0$ satisfies conditions {\rm (\ref{i.cond_1})}, {\rm (\ref{i.cond_2})} and {\rm (\ref{i.cond_3})}
of Theorem~\ref{t.conditions}.
\end{prop}

To begin, notice that if $G_0=(g_0,\dots,g_{\ell_0-1})$, $G_1=(g_{\ell_0},\dots,g_{\ell_0+\ell_1-1})$
and $G_2=(g_{\ell_0+\ell_1},\dots,g_{\ell_0+\ell_1+\ell_2-1})$ 
respectively satisfy conditions (\ref{i.cond_1}), (\ref{i.cond_2}) and (\ref{i.cond_3}),
then $G=(g_0,\dots, g_{\ell-1})$ with $\ell=\ell_0+\ell_1+\ell_2$ satisfies all three conditions. 
Therefore, to prove Proposition~\ref{p.open_conditions}, 
one can prove independently the existence of open sets satisfying each of the three conditions. 

Condition~(\ref{i.cond_3}) is trivially nonempty and $C^2$-open (actually $C^1$-open).

\begin{proof}[Proof that the maneuverability condition~{\rm (\ref{i.cond_2})} 
is nonempty and open.]
This is an easy compactness argument, but let us spell out the details for the reader's convenience.

Let $t=(t_1,\dots,t_d)\in\{-1,+1\}^d$, where $d=\dim M$.
For every $(x,\flag)\in\cF M$ there is $g\in \Diff^2(M)$ such that 
$$
t_i \, \log M_{i,i}(Dg(x),\flag) > 0 \quad \text{for each $i$.}
$$
By continuity, it follows that there is a neighborhood $U$ of $(x,\flag)$ in $\cF M$
and a neighborhood $\cV$ of $g$ in $\Diff^1(M)$ such that
such that if $(x',\flag') \in U$ and $g' \in \cV$ then 
$$
t_i \, \log M_{i,i}(Dg'(x'),\flag')  > 0 \quad \text{for each $i$.}
$$
By compactness, we can extract a finite subcover of $\cF M$ formed by neighborhoods of the type $U$,
and keep the associated $C^2$-diffeomorphisms.
Repeat the same procedure for each $t=(t_1,\dots,t_d)\in\{-1,+1\}^d$
and collect all diffeomorphisms.
This shows that Condition~{\rm (\ref{i.cond_2})} is nonempty and $C^2$-open (actually $C^1$-open). 
\end{proof}

In order to deal with the positive minimality condition~(\ref{i.cond_1}),
we start by proving the following criterion:

\begin{lemm}[Minimality criterion]\label{l.minimality} 
Let $N$ be a compact connected Riemannian manifold 
and let $H \subset \Homeo(N)$.
Assume that there exists a finite open cover $\{V_i\}$ of $N$
with the following properties:
\begin{itemize}
\item For every $i$ there exists a map $h_i \in \langle H \rangle$ 
whose restriction to $h_i^{-1}(V_i)$ is a (uniform) contraction.
\item The cover $\{V_i\}$ has a Lebesgue number $\delta$ such that
the orbit of every point $x\in N$ is $\delta$-dense in $N$.
\end{itemize}
Then the IFS generated by $H$ is positively minimal.
\end{lemm}

\begin{proof}
Since the cover $\{V_i\}$ is finite, there exists $\alpha$ with $0<\alpha<1$
such that each restriction $h_i | h_i^{-1}(V_i)$ is an $\alpha$-contraction.

\begin{clai}\label{cl.balls}
If $y \in N$, $r>0$, 
and $\overline{B(y,r)} \subset V_i$
then 
$h_i \big( B ( h_i^{-1}(y) , \alpha^{-1}r ) \big) \subset B(y,r)$.
\end{clai}

\begin{proof}[Proof of the claim]
Since $h_i | h_i^{-1}(V_i)$ is a (uniform) contraction, we have
$$
y \in V_i \quad \Rightarrow \quad 
h_i \big( B ( h_i^{-1}(y) , \alpha^{-1}r ) \big) \subset B(y,r) \cup (M \setminus V_i) \, .
$$
Since the ball $B ( h_i^{-1}(y) , \alpha^{-1}r )$ is connected\footnote{Here we use that $N$ is a connected Riemannian manifold (and not only a metric space).},
the claim follows.
\end{proof}

Now fix any point $x \in N$, and assume the orbit of $x$ is $\epsilon$-dense for some $\epsilon \le \delta$.
For any $y \in N$, we can find some $V_i$ containing $\overline{B(y,\alpha\epsilon)}$.
By the $\epsilon$-denseness of the orbit of $x$, there is $h \in \langle H \rangle$ such that
$h(x) \in B ( h_i^{-1}(y) , \epsilon )$.
It follows from Claim~\ref{cl.balls} that $h_i \circ h (x) \in B(y,\alpha\epsilon)$.
Since $y$ is arbitrary, this shows that the orbit of $x$ is $\alpha \epsilon$-dense.
By induction, this orbit is $\alpha^n\epsilon$-dense for any $n>0$;
so it is dense, as we wanted to show. 
\end{proof}

If $H$ is a finite set of \emph{diffeomorphisms} of $N$ satisfying the assumptions of Lemma~\ref{l.minimality}, 
then these assumptions are also satisfied for sufficiently small $C^1$-perturbations of the elements of $H$.
In other words, the hypotheses of Lemma~\ref{l.minimality} are $C^1$-robust.\footnote{Using this, it is easy to establish the existence of $C^1$-robustly positively minimal finitely generated IFS's on any compact connected manifold. More interestingly, Homburg~\cite{Homburg} shows that two generators suffice.}

Therefore, to show that 
the positive minimality condition~{\rm (\ref{i.cond_1})} 
has non-empty interior, we are reduced to show the following:

\begin{lemm} 
Given any compact connected manifold $M$, there is $\ell\in \NN^*$ 
and $g_0,\dots, g_{\ell-1}\in \Diff^2(M)$ such that the induced diffeomorphisms $\cF g_0$, \dots, $\cF g_{\ell-1} \in \Diff^1(\cF N)$ satisfy the hypotheses of Lemma~\ref{l.minimality}.
\end{lemm}

\begin{proof} 
This is another easy compactness argument.
	
For any $x\in M$ and any flag $\flag$ at $x$, there is a diffeomorphism $g\in \Diff^2(M)$ such that $x$ is a hyperbolic attracting fixed point of $g$ such that the eigenvalues of $Dg(y)$ have different moduli,
and moreover $\flag$ is the stable flag of $g$ at $x$. 
Then $(x,\flag)$ is a hyperbolic attracting fixed point of $\cF g$,
and in particular there is a open neighborhood $U{(x,\flag)}$  on which  $\cF g$ induces a uniform contraction.
Denote $V{(x,\flag)}=(\cF g)(U{(x,\flag)})$. The sets $V{(x,\flag)}$ form an open cover of $\cF M$, so that one can extract a finite subcover $V_0$, \dots, $V_{\ell_0-1}$. 
Let $g_0$, \dots, $g_{\ell_0-1}$ be the corresponding diffeomorphisms as above.

Given two points $(x,\flag)$, $(x',\flag') \in\cF M$,
it follows from the connectedness of $M$ that 
there is $g\in \Diff^2(M)$ such that $\cF g(x,\flag)=(x',\flag')$. 
A simple compactness argument shows that there exist $\ell_1$ and $g_{\ell_0}$, \dots, $g_{\ell_0+\ell_1-1}\in \Diff^2(M)$ such that for every $(x,\flag) \in\cF M$, the set
$\{\cF g_{\ell_0}(x,\flag),\dots, \cF g_{\ell_0+\ell_1-1}(x,\flag)\}$ is $\delta$-dense in $\cF M$,
where $\delta$ is a Lebesgue number of the cover $\{V_0, \dots, V_{\ell_0-1}\}$. 

Denote $\ell=\ell_0+\ell_1$. Now the action of $H = \{\cF g_0, \dots \cF g_{\ell-1}\}$ on $N = \cF M$ satisfies all the hypotheses of Lemma~\ref{l.minimality}, which allows us to conclude.
\end{proof}

This completes the proof of Proposition~\ref{p.open_conditions};
as explained in Section~\ref{s.conditions}, the 
main Theorems \ref{t.main} and \ref{t.mainflag} follow.

\begin{rema}\label{r.orient}
It is possible to adapt the proof of Theorem~\ref{t.mainflag}
for the \emph{oriented} flag bundle $\ordF M$,  
provided that the manifold $M$ is non-orientable
(basically because $\ordF M$ is then connected).
However, if $M$ is orientable, then the corresponding version of Theorem~\ref{t.mainflag} is 
false.
For example, since every diffeomorphism of $M = \CC P^2$ preserves orientation
(see \cite[p.~140]{Hirsch}), the induced action on $\ordF M$ fixes each of the two connected components,
and hence no $1$-step skew-product $\ordF \phi_G$ can be transitive.
\end{rema}

\section{Positive entropy}\label{s.entropy}

\begin{proof}[Proof of Theorem~\ref{t.C1easy}]
Let us say that an IFS $G = (g_0, \dots, g_{\ell-1}) \in (\Diff^1(M))^2$ 
has the \emph{bi-maneuverability property}
if $\ell$ is even and 
both IFS's 
\begin{equation}\label{e.partition}
G_0 = (g_0, \dots, g_{\ell/2-1}) \quad \text{and} \quad G_1 = (g_{\ell/2}, \dots, g_{\ell-1}).
\end{equation}
have the maneuverability property.

The set $\cV \subset (\Diff^1(M))^\ell$ of the IFS's with 
the bi-maneuverability property is nonempty and open,
provided $\ell$ is even and large enough;
this follows immediately from the analogous statements for maneuverability that we proved in 
Section~\ref{s.open}.
We will prove that the set $\cV$ satisfies the conclusions of the theorem.

Fix $G \in \cV$.
Fix any $(x_0, \flag_0) \in M \times \cF M$.
For any $\theta = (\theta_n)_n \in \{0,1\}^{\ZZ} = 2^{\ZZ}$, 
we will define a sequence $\omega(\theta) = (\omega_n(\theta))_n \in \ell^\ZZ$,
First we define the positive part of the sequence:
Assuming by induction that
$\omega_0(\theta)$, \dots, $\omega_{n-1}(\theta) \in \{ 0,\dots, \ell-1\}$ were already defined,
define 
$$
(x_n(\theta),\flag_n(\theta)) = \cF g_{[\omega_{n-1}(\theta) \dots \omega_0(\theta)]} (x_0,\flag_0).
$$
By the bi-maneuverability property, we can 
choose a symbol $\omega_n(\theta)$ satisfying the following properties:
\begin{itemize}
	\item for each $i \in \{1,\dots,d\}$, the number 
	$\log M_{i,i}(Dg_{\omega_n(\theta)} (x_n(\theta)),\flag_n(\theta))$
	is negative if $\log M_{i,i}(Dg_{[\omega_{n-1}(\theta) \dots \omega_0(\theta)]} (x_0), \flag_0)$ 
	is positive and positive otherwise.
	\item $\omega_n(\theta) < \ell/2$ if and only if $\theta_n = 0$.
\end{itemize}
This defines the positive part of the sequence $\omega(\theta)$.
The negative part is defined analogously, using the inverse maps.

The construction implies that 
$$
|\log M_{i,i} (Dg_{[\omega_m(\theta) \omega_{m-1}(\theta) \dots \omega_{n}(\theta)]} (x_n(\theta),\flag_n(\theta) ))| \le 2C, \quad \text{for all $m \ge n$ in $\ZZ$,}
$$
where $C$ is a constant with property \eqref{e.def_C}.

Define the following compact subset of $\ell^\ZZ \times \cF M$:
$$
\tilde \Lambda_G := \text{ closure of }
\Big\{ \big(\sigma^m(\omega(\theta)), x_m(\theta), \flag_m(\theta) \big) ;\; \theta\in 2^\ZZ, \ m \in \ZZ \Big\} .
$$
Let $\Lambda_G$ be the projection of $\tilde \Lambda_G$ in $\ell^\ZZ \times M$.

By continuity, we have
$$
\| M_{i,i} ((Dg_{[\omega_{n-1} \dots \omega_0]} (x, \flag) )) \|
\le C,
\quad \text{for all $(\omega,x,\flag) \in \tilde \Lambda_G$, $n\ge 0$, $i\in\{1,\dots,d\}$.}
$$

Let us check that $\Lambda_G$ has the zero exponents property (\ref{i.easy1}).
Let $\mu$ be an ergodic $\phi_G$-invariant measure whose support is contained in $\Lambda_G$.
Let $\nu$ be a lift of $\mu$ that is ergodic for $\cF \phi_G$.
It follows from the Ergodic Theorem that 
the Furtenberg vector $\vec\Lambda(\nu) = (\Lambda_1, \dots, \Lambda_d)$ defined by 
\eqref{e.furstenberg} is zero.
By Proposition~\ref{p.Lyap_permutation}, the fibered Lyapunov exponents of $\mu$ are all zero.

Finally, let us check the positive entropy property (\ref{i.easy2}).
We have the following commutative diagram:
$$
\begin{tikzcd}
\Lambda_G \arrow{r}{\phi_G} \arrow{d}{\pi} &\Lambda_G \arrow{d}{\pi}\\
2^\ZZ \arrow{r}{\sigma_2}                  &2^\ZZ
\end{tikzcd}
$$
where $\sigma_2$ is the $2$-shift, 
and 
$\pi(\omega,x) = \theta$ with $\theta_n = 1$ if and only if $\omega_n<\ell/2$.
Since $\pi$ is surjective, $h_\mathrm{top}(\phi_G|\Lambda_G) \ge h_\mathrm{top}(\sigma_2) = \log 2$.

This concludes the proof of Theorem~\ref{t.C1easy}.
\end{proof}

\begin{rema}
The $C^2$-open sets of IFS satisfying the conclusions of Theorem~\ref{t.main}
can be taken also satisfying the conclusions of Theorem~\ref{t.C1easy}:
it suffices to replace maneuverability by bi-maneuverability in the construction.
\end{rema}

\begin{rema}
As it is evident from its proof, Theorem~\ref{t.C1easy} has a flag bundle version.
\end{rema}

\begin{rema} 
Artur Avila suggested an alternative proof of Theorem~\ref{t.C1easy}
using ellipsoid bundles instead of flag bundles,
and obtaining compact sets $\Lambda_G$ where derivatives along orbits 
are uniformly bounded away from zero and infinity.
\end{rema}

\begin{ack}
We are grateful to the referee for some corrections.
\end{ack}
	

\end{document}